\documentclass[final,leqno]{siamltex}
\usepackage{amsmath}
\usepackage{graphicx}
\usepackage[notcite,notref]{showkeys}
\usepackage{mathrsfs}
\usepackage{float}
\usepackage{amsfonts,amssymb}
\usepackage{dsfont}
\usepackage{pifont}
\usepackage{hyperref}
\usepackage{multirow}

\newcounter{dummy} \numberwithin{dummy}{section}
\newtheorem{defi}[dummy]{Definition}
\newcommand{\bq}{\begin{equation}}
\newcommand{\eq}{\end{equation}}
\renewcommand{\ldots}{\dotsc}
\newtheorem{algorithm}{Weak Galerkin Algorithm}

\newcommand{\bu}{{\bf u}}
\newcommand{\bw}{{\bf w}}

\newcommand{\be}{{\bf e}}
\newcommand{\bv}{{\bf v}}
\newcommand{\bQ}{{\bf Q}}
\newcommand{\curl}{{\nabla\times}}
\newcommand{\cw}{{\nabla_w\times}}

\def\T{{\mathcal T}}
\def\E{{\mathcal E}}
\def\V{{\mathcal V}}
\def\W{{\mathcal W}}
\def\Q{{\mathbb Q}}
\def\pT{{\partial T}}
\def\l{{\langle}}
\def\r{{\rangle}}
\def\jump#1{{[\![#1[\!]}}
\def\bbf{{\bf f}}

\def\bn{{\bf n}}
\def\bq{{\bf q}}
\def\bt{{\bf t}}

\def\3bar{{|\hspace{-.02in}|\hspace{-.02in}|}}

\begin{document}

\title{A Weak Galerkin Finite Element Method for the Maxwell Equations}

\author{
Lin Mu\thanks{Department of Mathematics, Michigan State University,
      East Lansing, MI 48824 (linmu@ msu.edu)}
\and Junping Wang\thanks{Division of Mathematical Sciences, National
Science Foundation, Arlington, VA 22230, USA,
\texttt{jwang@nsf.gov}. The research of Wang was supported by the
NSF IR/D program, while working at the Foundation. However, any
opinion, finding, and conclusions or recommendations expressed in
this material are those of the author and do not necessarily reflect
the views of the National Science Foundation.} \and Xiu
Ye\thanks{Department of Mathematics, University of Arkansas at
Little Rock, Little Rock, AR 72204, USA, \texttt{xxye@ualr.edu}.
This author was supported in part by the National Science Foundation
under Grant No. DMS-1115097.} \and Shangyou Zhang\thanks{Department
of Mathematical Sciences, University of Delaware, Newark, DE 19716,
USA, (szhang@\ udel.edu)}}

\maketitle

\begin{abstract}
This paper introduces a numerical scheme for time harmonic Maxwell's
equations by using weak Galerkin (WG) finite element methods. The WG
finite element method is based on two operators: discrete weak curl
and discrete weak gradient, with appropriately defined
stabilizations that enforce a weak continuity of the approximating
functions. This WG method is highly flexible by allowing the use of
discontinuous approximating functions on arbitrary shape of
polyhedra and, at the same time, is parameter free. Optimal-order of
convergence is established for the weak Galerkin approximations in
various discrete norms which are either $H^1$-like or $L^2$ and
$L^2$-like. An effective implementation of the WG method is
developed through variable reduction by following a Schur-complement
approach, yielding a system of linear equations involving unknowns
associated with element boundaries only. Numerical results are
presented to confirm the theory of convergence.
\end{abstract}


\section{Introduction}

In this paper, we are concerned with new developments of numerical
methods for the time-harmonic Maxwell equations in a heterogeneous
medium $\Omega\subset \mathbb{R}^3$. The model problem seeks unknown
functions $\bu$ and $p$ satisfying
\begin{eqnarray}
\nabla\times(\mu\nabla\times \bu)-\epsilon\nabla p &=&{\bf f}_1\quad \mbox{in}\;\Omega,\label{moment1}\\
\nabla\cdot(\epsilon\bu)&=&g_1\quad\mbox{in}\;\Omega,\label{cont1}\\
\bu\times\bn &=& \phi\quad\mbox{on}\;\partial\Omega,\label{bcc1}\\
p&=&0\quad\mbox{on}\;\partial\Omega,\label{bc1}
\end{eqnarray}
where the coefficients $\mu>0$ and $\epsilon>0$ are the magnetic
permeability and the electric permittivity of the medium,
respectively.

A weak formulation for (\ref{moment1})-(\ref{bc1}) seeks $(\bu, p)
\in H(\hbox{curl};\Omega) \times H_0^1(\Omega)$ such that
$\bu\times\bn=\phi$ on $\partial\Omega$ and
\begin{eqnarray}
(\nu\nabla\times\bu,\ \nabla\times\bv)-(\bv,\nabla p)&=&({\bf f},\  \bv),
\quad \forall \bv \in H_0(\hbox{curl};\Omega)\label{w1}\\
(\bu,\nabla q)&=&-(g,q),\quad\forall q\in H_0^1(\Omega),\label{w2}
\end{eqnarray}
where $\nu=\mu/\epsilon$, ${\bf f}={\bf f}_1/\epsilon$ and
$g=g_1/\epsilon$.

The Maxwell equations have been studied extensively in literature by
using various numerical methodologies including
$H(\hbox{curl};\Omega)$-conforming edge element approaches
\cite{boss,jin,monk, Nedelec, Nedelec1} and discontinuous Galerkin
methods \cite{bls,bcnl, hps,hps-1,ps,psm}. Particularly in
\cite{hps-2}, a mixed DG formulation for the problem
(\ref{moment1})-(\ref{bc1}) was introduced and analyzed. In this DG
formulation, both $\bu$ and $p$ are approximated by piecewise
$[P_k(T)]^3$ and $P_k(T)$ functions if $T$ is a tetrahedron and by
piecewise $[Q_k(T)]^3$ and $Q_k(T)$ if $T$ is a parallelepiped,
where $P_k(T)$ denotes the set of polynomials of total degree $k$
and $Q_k(T)$ the set of polynomials of degree $k$ in each variable.

The weak Galerkin (WG) finite element method refers to a general
finite element technique for partial differential equations where
differential operators are approximated as discrete distributions or
discrete weak derivatives. The method was first introduced in
\cite{wy, wy-mixed} for second order elliptic equations, and was
later extended to other partial differential equations including the
Stokes equations \cite{wang-ye-stokes} and the biharmonic equation
\cite{mu-wang-ye-biharmonic, cwang-jwang-biharmonic}. The current
research indicates that the concept of discrete weak differential
operators offers a new paradigm in numerical methods for partial
differential equations.

In this paper, we apply the idea of weak Galerkin to the problem
(\ref{moment1})-(\ref{bc1}). In essence, this procedure shall
introduce a discrete curl operator, which shall be combined with the
discrete weak gradient as introduced in \cite{wy} to yield a finite
element scheme for the Maxwell equations. In this WG method, two
types of weak functions are used: $\bu_h=\{\bu_0,\bu_b\}\in
[P_s(T)]^3\times [P_t(e)]^3$ and $p_h=\{p_0,p_b\}\in P_\ell(T)\times
P_\iota(e)$, with $\bu_h=\bu_0$ and $p_h=p_0$ inside of each element
and $\bu_h=\bu_b$ and $p_h=p_b$ on the boundary of the element.
Error estimates of optimal order are established for the WG
approximations in appropriate norms for the case of $s=t=k$ and
$\ell=k-1$, $\iota=k$ with $k\ge 1$. For the case of $s=t=k$ and
$\ell=\iota=k-1$, only numerical experiments are conducted to
illustrate the performance of the corresponding WG finite element
scheme; a theoretical study of this WG method is left to interested
readers for an investigation.

The use of weak functions and weak derivatives makes the WG method
highly flexible on the construction of finite element functions on
partitions with arbitrary polygons or polyhedrons. Compared with the
DG method in \cite{hps-2}, our WG methods make use of additional
variables $\bu_b$ and $p_b$ defined on the boundary of the elements.
However, the variables $\bu_0$ and $p_0$ defined on each element can
be eliminated through a local process/computation, yielding a system
of linear equations involving only the variables $\bu_b$ and $p_b$.
Consequently, the WG method has much less number of globally coupled
unknowns than DG methods. In addition, the weak Galerkin finite
element method is parameter independent in its stability and
convergence.

The paper is organized as follows. In Section
\ref{Section:preliminaries}, we introduce some basic notations. In
Section \ref{Section:wg-wd}, we discuss some discrete weak
differential operators, particularly a discrete weak curl. Section
\ref{Section:WGFEM} is devoted to a presentation of the weak
Galerkin finite element scheme for the problem
(\ref{w1})-(\ref{w2}). In Section \ref{Section:ErrorEquation}, we
derive an error equation for the WG finite element approximation. In
Section \ref{Section:L2Projections}, we introduce two types of $L^2$
projection operators and derive some estimates for them. Sections
\ref{Section:H1ErrorEstimates} and \ref{Section:L2ErrorEstimates}
are devoted to an error analysis for the WG finite element
approximations. In Section \ref{Section:Schur}, we discuss an
efficient implementation method by using variable
reductions/elimination. Finally in Section \ref{Section:NE}, we
present some numerical results that verify the theory established in
the previous sections.

\section{Preliminaries and Notations}\label{Section:preliminaries}

Let $D$ be any open bounded domain with Lipschitz continuous
boundary in $\mathbb{R}^3$. We use the standard definition for the
Sobolev space $H^s(D)$ and their associated inner products
$(\cdot,\cdot)_{s,D}$, norms $\|\cdot\|_{s,D}$, and seminorms
$|\cdot|_{s,D}$ for any $s\ge 0$. For example, for any integer $s\ge
0$, the seminorm $|\cdot|_{s, D}$ is given by
$$
|v|_{s, D} = \left( \sum_{|\alpha|=s} \int_D |\partial^\alpha v|^2
dD \right)^{\frac12}
$$
with the usual notation
$$
\alpha=(\alpha_1, \ldots, \alpha_d), \quad |\alpha| =
\alpha_1+\ldots+\alpha_d,\quad
\partial^\alpha =\prod_{j=1}^3\partial_{x_j}^{\alpha_j}.
$$
The Sobolev norm $\|\cdot\|_{m,D}$ is given by
$$
\|v\|_{m, D} = \left(\sum_{j=0}^m |v|^2_{j,D} \right)^{\frac12}.
$$

The space $H^0(D)$ coincides with $L^2(D)$, for which the norm and
the inner product are denoted by $\|\cdot \|_{D}$ and
$(\cdot,\cdot)_{D}$, respectively. When $D=\Omega$, we shall drop
the subscript $D$ in the norm and inner product notation.

The space $H(\hbox{curl};D)$ is defined as the set of vector-valued
functions on $D$ which, together with their curl, are square
integrable; i.e.,
\[
H({\rm curl}; D)=\left\{ \bv: \ \bv\in [L^2(D)]^3, \nabla\times\bv \in
[L^2(D)]^3\right\}.
\]

\section{Weak Derivatives}\label{Section:wg-wd}

The two differential operators used in (\ref{w1}) and (\ref{w2}) are
curl and gradient operators. The goal of this section is to
introduce an analogy of the curl and gradient operator, called weak
curl and weak gradient operators, when the applied functions are
discontinuous.

\subsection{Weak gradient and discrete weak gradient}

The concept of weak gradient and its discrete analogue was
introduced in \cite{wy}. This subsection is presented for the sake
of completeness of presentation.

Let $K$ be any polyhedral domain with boundary $\partial K$. A  weak
function on the region $K$ refers to a function $v=\{v_0, v_b\}$
such that $v_0\in L^2(K)$ and $v_b\in L^2(\partial K)$. The first
component $v_0$ can be understood as the value of $v$ in $K$, and
the second component $v_b$ represents $v$ on the boundary of $K$.
Note that $v_b$ may not necessarily be related to the trace of
$\bv_0$ on $\partial K$ should a trace be well-defined. Denote by
$\W(K)$ the space of weak functions on $K$; i.e.,
\begin{equation}\label{hi.888}
\W(K):= \{v=\{v_0, v_b \}:\ v_0\in L^2(K),\; v_b\in L^2(\partial
K)\}.
\end{equation}
The weak gradient operator is defined as follows.
\medskip

\begin{defi} (Weak Gradient)
The dual of $L^2(K)$ can be identified with itself by using the
standard $L^2$ inner product as the action of linear functionals.
With a similar interpretation, for any $v\in \W(K)$, the weak
gradient of $v$ is defined as a linear functional $\nabla_w v$ in
the dual space of $[H^1(K)]^3$ whose action on each $q\in
[H^1(K)]^3$ is given by
\begin{equation}\label{wg}
(\nabla_w v, q)_K := -(v_0, \nabla\cdot q)_K + \langle v_b,
q\cdot\bn\rangle_{\partial K},
\end{equation}
where $\bn$ is the outward normal direction to $\partial K$,
$(v_0,\nabla\cdot q)_K=\int_K v_0 (\nabla\cdot q)dK$ is the $L^2$
inner product of $v_0$ and $\nabla\cdot q$, and $\langle v_b,
q\cdot\bn\rangle_{\partial K}$ is the $L^2$ inner product of
$q\cdot\bn$ and $v_b$ in $L^2(\partial K)$.
\end{defi}

The Sobolev space $H^1(K)$ can be embedded into the space $\W(K)$ by
an inclusion map $i_\W: \ H^1(K)\to \W(K)$ defined as follows
$$
i_\W(\phi) = \{\phi|_{K}, \phi|_{\partial K}\},\qquad \phi\in H^1(K).
$$
With the help of the inclusion map $i_\W$, the Sobolev space $H^1(K)$
can be viewed as a subspace of $\W(K)$ by identifying each $\phi\in
H^1(K)$ with $i_\W(\phi)$.

\medskip
Let $P_{r}(K)$ be the set of polynomials on $K$ with degree no more
than $r$.

\begin{defi}
(Discrete Weak Gradient) The discrete weak gradient operator, denoted by
$\nabla_{w,r, K}$, is defined as the unique polynomial
$(\nabla_{w,r, K}v) \in [P_r(K)]^3$ satisfying the following
equation
\begin{equation}\label{d-g}
(\nabla_{w,r, K}v, q)_K = -(v_0,\nabla\cdot q)_K+ \langle v_b,
q\cdot\bn\rangle_{\partial K},\qquad \forall q\in [P_r(K)]^d.
\end{equation}
\end{defi}

\subsection{Weak curl and discrete weak curl}\label{Section:WeakCurl}

To define weak curl, we require weak functions $\bv=\{\bv_0,
\bv_b\}$ such that $\bv_0\in [L^2(K)]^3$ and $\bv_b\times\bn\in
[L^2(\partial K)]^3$. The first component $\bv_0$ can be understood
as the value of $\bv$ in $K$. The second component $\bv_b$
represents the value of $\bv$ on the boundary of $K$.

Denote by $\V(K)$ the space of vector-valued weak functions on $K$;
i.e.,
\begin{equation}\label{hi.999}
\V(K) = \{\bv=\{\bv_0, \bv_b \}:\ \bv_0\in [L^2(K)]^3,\;
\bv_b\times\bn\in [L^{2}(\partial K)]^3\}.
\end{equation}
Then, we define a  weak curl operator as follows.
\medskip

\begin{defi}
(Weak Curl) The dual of $[L^2(K)]^3$ can be identified with itself by using the
standard $L^2$ inner product as the action of linear functionals.
With a similar interpretation, for any $\bv\in \V(K)$, the  weak
curl of $\bv$ is defined as a linear functional $\nabla_w
\times\bv$ in the dual space of $[H^1(K)]^3$ whose action on each
$\varphi\in [H^1(K)]^3$ is given by
\begin{equation}\label{w-c}
(\nabla_w\times\bv, \varphi)_K := (\bv_0, \nabla\times\varphi)_K
+\langle \bv_b\times\bn, \varphi\rangle_{\partial K},
\end{equation}
where $\bn$ is the outward normal direction to $\partial K$,
$(\bv_0,\nabla\times\varphi)_K=\int_K \bv_0\cdot\nabla\times\varphi
dK$ is the $L^2$ inner product of $\bv_0$ and $\nabla\times\varphi$,
and $\langle \bv_b\times\bn, \varphi\rangle_{\partial K}$ is the
inner product in $L^2(\partial K)$.
\end{defi}

The Sobolev space $[H^1(K)]^3$ can be embedded into the space $\V(K)$
by an inclusion map $i_V: \ [H^1(K)]^3\to \V(K)$ defined as follows
$$
i_\V(\phi) =\{\phi|_{K}, \phi|_{\partial K}\},\qquad \phi\in [H^1(K)]^3.
$$
Let $K$ be any  polyhedral domain with boundary $\partial K$. For
each face $e\in\partial K$, let $\bt_1$ and $\bt_2$ be two assigned
unit vectors on the face $e$ such that $\bt_1$, $\bt_2$ and $\bn$
are orthogonal each other. Thus, we have
$\bv_b|_e=v_1\bt_1+v_2\bt_2+v_n\bn$. Define
$\bar{\bv}_b=v_1\bt_1+v_2\bt_2$.
Obviously,$\bar{\bv}_b\times\bn=\bv_b\times\bn$. Since the quantity
of interest is not $\bv_b$ but $\bv_b\times\bn$, we will let
$\bv_b=\bar{\bv}_b$ in order to reduce the number of the unknowns.
\medskip

\begin{defi} (Discrete Weak Curl)
For a given $K$, a discrete weak curl operator,
denoted by $\nabla_{w,r,K}\times$, is defined as the unique polynomial
$(\nabla_{w,r.K}\times\bv) \in [P_r(K)]^3$ that satisfies the following
equation
\begin{equation}\label{d-c}
  (\nabla_{w,r,K}\times\bv, \varphi)_K := (\bv_0, \nabla\times\varphi)_K
 + \langle \bv_b\times\bn, \varphi\rangle_{\partial K},\qquad
   \forall \varphi\in [P_r(K)]^3.
\end{equation}
\end{defi}

\section{Numerical Algorithms}\label{Section:WGFEM}

Let ${\cal T}_h$ be a partition of the domain $\Omega$ with mesh
size $h$ that consists of polyhedra of arbitrary shape. Assume that
the partition ${\cal T}_h$ is shape regular in the sense as defined
in \cite{wy-mixed}; i.e. ${\cal T}_h$ satisfies a set of conditions
given in \cite{wy-mixed}. Denote by ${\cal E}_h$ the set of all
faces in ${\cal T}_h$, and let ${\cal E}_h^0={\cal
E}_h\backslash\partial\Omega$ be the set of all interior faces.

Let $e\in\E_h^0$ be shared by two elements $T_1$ and $T_2$. Let
$\bt_1$ and $\bt_2$ be two tangential unit vectors on face $e\in
\E_h$. For $k\ge 1$, define a weak Galerkin finite element spaces
associated with $\T_h$ as
\begin{align} \label{Vh}
V_h =\Big\{ \bv=\{\bv_0, \bv_b=v_1\bt_1+v_2\bt_2\}   : & \ \bv_0|_{T}\in
      [P_k(T)]^3, \\
     & v_1, v_2\in  P_k(e),\ e\subset\pT\Big\}, \nonumber
\end{align}
and
\begin{align} \label{Wh}
W_h =\Big\{w=\{w_0, w_b\} : \quad & \{w_0, w_b\}|_{T}\in
P_{k-1}(T)\times
P_{k}(e), \ e\subset\pT, \\
  & w_b=0 \ {\rm on}\ \partial\Omega\Big\}.\nonumber
\end{align}
We also introduce the following subspace of $V_h$,
\[
V_{h,0} =\left\{ \bv=\{\bv_0, \bv_b\}\in V_h, \;
\bv_b\times\bn|_e=0,\ e\subset\partial\Omega \right\}.
\]
The discrete weak gradient $\nabla_{w,k-1}$ and the discrete weak
curl $\nabla_{w,k}\times$ on the finite element spaces $W_h$ and
$V_h$ can be computed by using (\ref{d-g}) and (\ref{d-c}) on each
element $T$ respectively; i.e.,
\begin{eqnarray*}
(\nabla_{w,k}v)|_T &=&\nabla_{w,k, T} (v|_T),\qquad \forall v\in
W_h\\
(\nabla_{w,k-1}\times\bv)|_T &=&\nabla_{w,k-1, T}\times (\bv|_T),\qquad \forall \bv\in V_h.
\end{eqnarray*}
For simplicity of notation, from now on we shall drop the subscript
$k$ in  $\nabla_{w,k}$ and $k-1$ in $\nabla_{w,k-1}\times$ for the
discrete weak gradient and the discrete weak curl.

Corresponding to the bilinear forms in (\ref{w1})-(\ref{w2}), we
introduce the following bilinear forms:
\begin{eqnarray*}
(\nu\nabla_w\times\bv,\ \nabla_w\times \bw)_h&=&\sum_{T\in\T_h}(\nu\nabla_w\times \bv,\ \nabla_w\times \bw)_T\\
(\bv,\ \nabla_w q)_h&=&\sum_{T\in\T_h}(\nabla_w q,\ \bv)_T.
\end{eqnarray*}
Furthermore, we stabilize the first one by adding an appropriate
stabilization term as follows:
\begin{eqnarray}\label{Stabilized-A-form}
a(\bv,\ \bw)&=&(\nu\nabla_w\times\bv,\
\nabla_w\times\bw)_h+s_1(\bv,\bw),
\end{eqnarray}
where
\begin{eqnarray}\label{Stabilization-T1}
s_1(\bv,\;\bw) &=& \sum_{T\in {\cal T}_h}h^{-1}\l
(\bv_0-\bv_b)\times\bn,\;\;(\bw_0-\bw_b)\times\bn\r_\pT.
\end{eqnarray}
For simplicity of notation, we introduce the following notation
\begin{eqnarray}\label{B-form}
b(\bv,\ q)&=&(\bv_0,\nabla_w q)_h
\end{eqnarray}
and a second stabilization term
\begin{eqnarray}\label{Stabilization-T2}
s_2(p,\;q) & = & \sum_{T\in {\cal T}_h}h\l
p_0-p_b,\;\;q_0-q_b\r_\pT.
\end{eqnarray}

\begin{algorithm}
Find $\bu_h=\{\bu_0,\bu_b\}\in V_h$  and
$p_h=\{p_0,p_b\}\in W_h$ satisfying  $\bu_b\times \bn= Q_b\phi$ on
$\partial \Omega$ and
\begin{eqnarray}
a(\bu_h,\ \bv)-b(\bv,\;p_h)&=&({\bf f},\;\bv_0),\quad\forall\ \bv=\{\bv_0,\; \bv_b\}\in
V_{h,0},\label{wg1}\\
b(\bu_h,\;q)+s_2(p_h,q)&=&-(g,q_0), \quad\forall\ q=\{q_0,\; q_b\}\in
W_h,\label{wg2}
\end{eqnarray}
where $Q_b \phi$ is an approximation of the boundary value in the
polynomial space $[P_k(\partial T\cap \partial\Omega)]^3$. For
simplicity, one may take $Q_b \phi$ as the standard $L^2$ projection
of the boundary value $\phi$ on each boundary segment.
\end{algorithm}
\smallskip

\begin{lemma}\label{lemma-ue}
The weak Galerkin finite element algorithm (\ref{wg1})-(\ref{wg2})
has a unique solution.
\end{lemma}

\smallskip

\begin{proof}
It suffices to show that zero is the only solution of
(\ref{wg1})-(\ref{wg2}) if $\bbf=0, \phi=0,$ and $g=0$. To this end,
assume that the homogeneous conditions are given. Take $\bv=\bu_h$
and $q=p_h$ in (\ref{wg1})-(\ref{wg2}). By adding the two resulting
equations, we obtain
\begin{eqnarray*}
 (\nu\cw\bu_h,\ \cw\bu_h)_h
   & + &\sum_{T\in\T_h}h^{-1}\l(\bu_0-\bu_b)\times\bn,\
(\bu_0-\bu_b)\times\bn\r_\pT\\
    & + & \sum_{T\in\T_h}h\l p_0-p_b,\ p_0-p_b\r_\pT=0,
\end{eqnarray*}
which implies $\nabla_w\times\bu_h=0$ on each $T$,
$\bu_0\times\bn=\bu_b\times\bn$ and $p_0=p_b$ on $\pT$. Note that
the boundary condition implies $\bu_b\times\bn=0$ on each
$e\subset\partial\Omega$. Then, it follows from (\ref{d-c}) and the
integration by parts that for any $\bv\in [P_{k-1}(T)]^3$
\begin{eqnarray*}
0&=&(\cw\bu_h,\bv)_T\\
&=&(\bu_0,\ \curl\bv)_T+\l\bu_b\times\bn,\ \bv\r_\pT\\
&=&(\curl\bu_0,\ \bv)_T+\l(\bu_b-\bu_0)\times\bn,\ \bv\r_\pT\\
&=&(\curl\bu_0,\ \bv)_T,
\end{eqnarray*}
which gives $\curl\bu_0=0$ on each $T\in\T_h$. Using (\ref{wg2}),
(\ref{d-g}) and the integration by parts, we have
\begin{eqnarray*}
0&=&\sum_{T\in\T_h}(\bu_0,\nabla_w
q)_T=-\sum_{T\in\T_h}(\nabla\cdot\bu_0,\
q_0)_T+\sum_{T\in\T_h}\l\bu_0\cdot\bn,\ q_b\r_\pT.
\end{eqnarray*}
Letting $q_0=\nabla\cdot\bu_0$ and $q_b=0$ in the above equation
yield $\nabla\cdot \bu_0=0$ on each $T\in\T_h$. Next, by letting
$q_0=0$ and $q_b$ be the jump of $\bu_0\cdot\bn$ on each interior
face $e$, we conclude that $\bu_0$ is continuous across each
interior face $e$ in the normal direction.

Note that $\nabla\times \bu_0=0$. Thus, there exists a potential
function $\phi$ such that $\bu_0=\nabla\phi$ on $\Omega$. It follows
from $\nabla\cdot \bu_0=0$ and the fact
that $\bu_0\cdot\bn$ is continuous
that $\Delta\phi=0$ is strongly satisfied in $\Omega$. The boundary
condition of (\ref{bcc1}) implies that
$\bu_0\times\bn=\nabla\phi\times\bn=0$ on $\partial\Omega$.
Therefore, $\phi$ must be a constant on $\partial\Omega$. The
uniqueness of the solution of the Laplace equation implies that
$\phi=const$ is the only solution of $\Delta\phi=0$ if $\Omega$ is
simply connected. Then we must have $\bu_0=\nabla\phi=0$. Since
$\bu_b\times\bn=\bu_0\times\bn=0$, we have $\bu_b=0$.

Since $\bu_h=0$, we then have $b(\bv,p_h)=0$ for any $\bv\in
V_{h,0}$. It follows from the definition of $b(\cdot,\cdot)$ and
$\nabla_w$ that
\begin{eqnarray}\label{eu1}
0&=&b(\bv,\ p_h)=(\bv_0,\nabla_w
p_h)_h\\
&=&-\sum_{T\in\T_h}(\nabla\cdot\bv_0,\ p_0)_T+\sum_{T\in\T_h}\l
\bv_0\cdot\bn,\ p_b\r_\pT\nonumber\\
&=&\sum_{T\in\T_h}(\bv_0,\nabla p_0)_T,\nonumber
\end{eqnarray}
where we have used the fact that $p_0=p_b$ on $\partial T$. Letting
$\bv=\{\bv_0,\bv_b\}=\{\nabla p_0, 0\}$ in (\ref{eu1}) gives $\nabla
p_0=0$ on each $T\in \T_h$, i.e. $p_0$ is a constant on $T\in\T_h$. Using the facts $p_0=p_b$
and  $p_b=0$ on $\partial\Omega$, we obtain $p_h=0$.
\end{proof}

\section{Error Equations}\label{Section:ErrorEquation}

For each element $T\in \T_h$, denote by $\bQ_0$ and $Q_0$ the $L^2$
projections onto $[P_k(T)]^3$ and $P_{k-1}(T)$ respectively. Let
$Q_b$  be the $L^2$ projection onto $P_k(e)$. Then we can define two
projections onto the finite element space $V_h$ and $W_h$ such that
on each element $T$,
$$\bQ_h\bv=\{\bQ_0\bv,Q_b\bv=Q_b(v_1)\bt_1+Q_b(v_2)\bt_2\},\quad  Q_hq=\{Q_0q,Q_bq\}.$$
In addition, denote by $\Q_h$ the local $L^2$ projection onto
$[P_{k-1}(T)]^3$. The projection operators $\Q_h$, $Q_h$ and $\bQ_h$
have some useful properties as stated in the following Lemma.

\begin{lemma}\label{lem-0} Let $\bQ_h=\{\bQ_0, Q_b\}$ and
$Q_h=\{Q_0, Q_b\}$ be the projection operators onto the finite
element spaces $V_h$ and $W_h$ respectively. Then, we have
\begin{equation}\label{key}
\cw (\bQ_h \bu) = \Q_h (\curl\bu)\qquad \forall \bu\in
H(curl;\Omega)
\end{equation}
and
\begin{equation}\label{key11}
\nabla_w(Q_h q) = \bQ_0(\nabla q)\qquad \forall q\in H^1(\Omega).
\end{equation}
\end{lemma}
\begin{proof}
Using (\ref{d-c}), the integration by parts, and the definition of
$\bQ_h$ and $\Q_h$, we have
\begin{eqnarray*}
(\cw (\bQ_h \bu),\; \bw)_T &=& (\bQ_0\bu,\; \curl\bw)_T + \langle
(Q_b\bu)\times\bn,\; \bw\rangle_{\pT}\\
&=&(\bu,\; \curl\bw)_T + \langle \bu\times\bn,\; \bw\rangle_{\partial T}\\
&=&(\curl\bu,\; \bw)_T=(\Q_h(\curl \bu),\; \bw)_T
\end{eqnarray*}
for any $\bw\in [P_{k-1}(T)]^3$. This implies that (\ref{key}) holds
true.

As to (\ref{key11}), we use the definition of $Q_h$ and the discrete
gradient operator $\nabla_w$ to obtain
\begin{eqnarray*}
(\nabla_w (Q_h p),\; \bv)_T &=& -(Q_0p,\; \nabla\cdot\bv)_T +
\langle
Q_b p,\; \bv\cdot\bn\rangle_{\pT}\\
&=&-(p,\; \nabla\cdot\bv)_T + \langle p,\; \bv\cdot\bn\rangle_{\partial T}\\
&=&(\nabla p,\; \bv)_T=(\bQ_0(\nabla p),\; \bv)_T
\end{eqnarray*}
for all $\bv\in [P_k(T)]^3$, which verifies the desired relation
(\ref{key11}).
\end{proof}

\medskip

Define two error functions as follows
\begin{eqnarray}\label{error-u}
\be_h&=&\{\be_0,\;\be_b\}=\{\bQ_0\bu-\bu_0,\;Q_b\bu-\bu_b\},\\
\varepsilon_h&=&\{\varepsilon_0,\;\varepsilon_b\}=\{Q_0
p-p_0,\;Q_bp-p_b\}.\label{error-p}
\end{eqnarray}
The rest of this section is to derive some equations that the above
error functions must satisfy. For simplicity of analysis, we assume
that the coefficient $\nu$ in (\ref{w1}) is a piecewise constant
function with respect to the finite element partition $\T_h$.

\begin{lemma}\label{Lemma:error-equation}
Let $(\bu_h; p_h)$ be the WG finite element solution arising from
(\ref{wg1}) and (\ref{wg2}), and $(\be_h; \varepsilon_h)$ be the
error between the WG finite element solution and the $L^2$
projection of the exact solution as defined in
(\ref{error-u})-(\ref{error-p}). Then, the following equations are
satisfied
\begin{eqnarray}
a(\be_h,\ \bv)-b(\bv,\ \epsilon_h)&=&\varphi_\bu(\bv)
\quad\;\;\forall\bv\in V_{h,0},\label{ee1}\\
b(\be_h,\ q)+s_2(\epsilon_h,\ q)&=&\phi_{\bu,p}(q)\quad\forall q\in
W_h,\label{ee2}
\end{eqnarray}
where
\begin{eqnarray}\label{varphi-u-v}
\varphi_\bu(\bv) &=&s_1(\bQ_h\bu,\ \bv)-l_1(\bu,\ \bv),\\
\phi_{\bu,p}(q) &=& s_2(Q_hp,\ q)+l_2(\bu,q),\label{phi-up-q}
\end{eqnarray}
and
\begin{eqnarray}
l_1(\bu,\ \bv)&=&\sum_{T\in\T_h}\l(I-\Q_h)\curl\bu, \ \nu(\bv_b-\bv_0)\times\bn \r_\pT\label{l1}\\
l_2(\bu,\ q)&=&\sum_{T\in\T_h}\langle q_0-q_b,\
(\bu-\bQ_0\bu)\cdot\bn\rangle_\pT.\label{l2}
\end{eqnarray}
\end{lemma}

\begin{proof}
Using (\ref{key}), (\ref{d-c}), and the integration by parts we have
\begin{eqnarray} \label{m1}
&  & (\nu\cw (\bQ_h\bu),\;\cw \bv)_T \\
&=&(\nu\Q_h(\curl\bu),\;\cw\bv)_T\nonumber\\
&=&(\nu\bv_0,\ \curl \Q_h(\curl \bu))_T+\langle \nu\bv_b\times\bn,
 \  \Q_h(\curl\bu)\rangle_\pT\nonumber\\
&=&(\nu\curl\bv_0,\; \Q_h(\curl \bu))_T+\langle \nu(\bv_b-\bv_0)\times\bn,
  \ \Q_h(\curl\bu)\rangle_\pT\nonumber\\
&=&(\nu\curl\bu,\;\curl\bv_0)_T+\l
\Q_h(\curl\bu),\nu(\bv_b-\bv_0)\times\bn\r_\pT.\nonumber
\end{eqnarray}
It follows from (\ref{key11}) that
\begin{eqnarray}
(\nabla_w(Q_h p),\bv_0)_T=(\bQ_0\nabla p,\bv_0)_T=(\nabla p,\bv_0)_T.\label{m2}
\end{eqnarray}
Next, using the definition of $\nabla_w$ and $\bQ_0$, we obtain
\begin{eqnarray}
(\bQ_0\bu,\ \nabla_wq)_T&=&-(q_0, \nabla\cdot (\bQ_0\bu))_T+\l q_b,\;\bQ_0\bu\cdot\bn\r_\pT\label{m33}\\
&=&(\nabla q_0,\ \bu)_T-\langle q_0-q_b,\ \bQ_0\bu\cdot\bn\rangle_\pT.\nonumber
\end{eqnarray}
Testing (\ref{moment1}) by $\bv_0$ with $\bv=\{\bv_0,\;\bv_b\}\in V_{h,0}$ gives
\begin{equation}\label{m3}
(\curl(\nu\curl\bu),\;\bv_0)- (\nabla p,\ \bv_0)=(\bbf,\; \bv_0).
\end{equation}
It follows from the integration by parts that
\[
(\curl(\nu\curl\bu),\;\bv_0)=\sum_{T\in\T_h}(\nu\curl\bu,\
\curl\bv_0)_T+\sum_{T\in\T_h}\l\nu(\bv_b-\bv_0)\times\bn,\
\curl\bu\r_\pT,
\]
where we use the fact that $\sum_{T\in\T_h}\langle\bv_b\times\bn, \
\nu\curl\bu\rangle_\pT=0$. Using (\ref{m1}) and the equation above,
we have
\begin{eqnarray} \label{m4}
(\curl(\nu\curl\bu),\;\bv_0)&=&(\nu\cw (\bQ_h\bu),\ \cw\bv)_h\\
& &\quad + \sum_{T\in\T_h}\l (I-\Q_h)\curl\bu,\
\nu(\bv_b-\bv_0)\times\bn\r_\pT. \nonumber
\end{eqnarray}
Substituting (\ref{m2}) and (\ref{m4}) into (\ref{m3}) yields
\[
(\nu\cw (\bQ_h\bu),\ \cw\bv)_h-(\nabla_wQ_h p,\bv_0)_h=(\bbf,\ \bv_0)-l_1(\bv,\ \bu).
\]
Adding $s_1(\bQ_h\bu,\ \bv)$ to both sides of the equation above gives
\begin{equation}\label{m5}
a(\bQ_h\bu,\ \bv)-b(\bv,\ Q_hp)=(\bbf,\;\bv_0)+\varphi_\bu(\bv).
\end{equation}

To derive a second equation, we test equation (\ref{cont1}) by $q_0$
with $q=\{q_0,q_b\}\in W_h$ and then use the integration by parts to
obtain
\begin{equation}\label{m41}
-\sum_{T\in\T_h}(\bu,\;\nabla q_0)_T+ \sum_{T\in\T_h}\l \bu\cdot\bn,\
q_0-q_b\r_{\pT}=(g,\; q_0),
\end{equation}
where we have used the fact $\sum_{T\in\T_h}\langle \bu\cdot\bn, \
q_b\rangle_\pT=0$. Combining (\ref{m33}) with (\ref{m41}) gives
\[
\sum_{T\in\T_h}(\bQ_0\bu,\;\nabla_w q)_T=-(g,\; q_0)+l_2(\bu,q).
\]
Adding $s_2(Q_hp,\ q)$ to both sides of the equation above gives
\begin{equation}\label{m6}
b(\bQ_h\bu,q)+s_2(Q_hp,\ q)=-(g,\; q_0)+\phi_{\bu,p}(q).
\end{equation}
Finally, the differences of (\ref{m5}) and (\ref{wg1}), (\ref{m6})
and (\ref{wg2}) yield the error equations (\ref{ee1}) and
(\ref{ee2}), respectively.
\end{proof}

\section{Preparation for Error Estimates}\label{Section:L2Projections}

For $\bv=\{\bv_0,\bv_b\}\in V_{h,0}$, define $\3bar\bv\3bar$ as
follows
\begin{equation}\label{3barnorm}
\3bar \bv\3bar^2=a(\bv,\;\bv)=\sum_{T\in\T_h}\nu\|\cw\bv\|_T^2
   +\sum_{T\in\T_h}h^{-1}\|(\bv_0-\bv_b)\times\bn\|_\pT^2.
\end{equation}
It is clear that $\3bar\cdot\3bar$ defines merely a semi-norm for
the linear space $V_{h,0}$. A norm can be derived from the semi-norm
$\3bar\bv\3bar$ by adding two more terms given as follows
\begin{equation}\label{3bar1nomr}
\3bar\bv\3bar_1 = \3bar\bv\3bar + \left(\sum_{T\in\T_h}
\|\nabla\cdot\bv_0\|_T^2
\right)^{\frac12}+\left(\sum_{e\in\E_h^0}h^{-1}\|\jump{\bv_0\cdot\bn}\|_e^2
\right)^{\frac12},
\end{equation}
where $\jump{\bv_0\cdot\bn}$ is the jump of the function $\bv_0$ at
each edge/face in the normal direction. The proof of Lemma
\ref{lemma-ue} can be employed to verify that $\3bar\cdot\3bar_1$ is
indeed a norm in $V_{h,0}$. For convenience, we also use the
following notation:
\begin{equation}\label{Semi-Norm-v1h}
|\bv|_{1,h}:=
\left(\sum_{T\in\T_h}h^{-1}\|(\bv_0-\bv_b)\times\bn\|^2_{\pT}\right)^{1/2}.
\end{equation}

The linear space $W_h$ can be equipped with the following norm
$$
\3bar q\3bar_0 = |q|_{0,h} + h \|\nabla q\|_{0,h},
$$
where
$$
|q|_{0,h}^2=\sum_{T\in\T_h}h\|q_0-q_b\|_\pT^2
$$
and
\begin{equation}\label{q-h1-seminorm}
\|\nabla q\|_{0,h} = \left(\sum_{T\in\T_h} \|\nabla
q_0\|_{T}^2\right)^{\frac12}
\end{equation}
for any $q\in W_h$.

\smallskip

The following Lemma provides some approximation estimates for the
projections $\bQ_h$, $\Q_h$, and $Q_h$.

\begin{lemma}\label{lem-1}
Let $\T_h$ be a WG shape regular partition of $\Omega$, $\bw\in
[H^{t+1}(\Omega)]^3$, $\rho\in H^t(\Omega)$, and $0\le t\le k$.
Then, for $0\le s\le 1$, we have
\begin{eqnarray}
&&\sum_{T\in\T_h} h_T^{2s}\|\bw-\bQ_0\bw\|_{s,T}^2\le C h^{2(t+1)}
\|\bw\|^2_{t+1},\label{Qh}\\
&&\sum_{T\in\T_h} h_T^{2s}\|\curl\bw-\Q_h(\curl\bw)\|^2_{s,T} \le
Ch^{2t}
\|\bw\|^2_{t+1},\label{Rh}\\
&&\sum_{T\in\T_h} h_T^{2s}\|\rho-Q_0\rho\|^2_{s,T} \le
Ch^{2t}\|\rho\|^2_{t}.\label{Lh}
\end{eqnarray}
\end{lemma}
Since the mesh $\T_h$ is assumed to be very general, the proof of
Lemma \ref{lem-1} is rather technical and can be found in
\cite{wy-mixed}.

\smallskip

Let $K$ be an element with $e$ as a face.  For any function $g\in
H^1(K)$, the following trace inequality has been proved for
arbitrary polyhedra  $K$ in \cite{wy-mixed}.
\begin{equation}\label{trace}
\|g\|_{e}^2 \leq C \left( h_K^{-1} \|g\|_K^2 + h_K \|\nabla
g\|_{K}^2\right).
\end{equation}
In particular, if $\xi$ is a polynomial on $K$, then the standard
inverse inequality can be applied to yield
\begin{equation}\label{trace-poly}
\|\xi\|_{e}^2 \leq C h_K^{-1} \|\xi\|_K^2.
\end{equation}
Using (\ref{trace}) and Lemma \ref{lem-1}, we can prove the
following result.

\begin{lemma}\label{Lemma:myestimates}
Let $\bw\in [H^{t+1}(\Omega)]^3$ and $p\in H^t(\Omega)$ and $\bv\in
V_h$ with $\frac12 <t\le k$. Then
\begin{eqnarray}
|s_1(\bQ_h\bw,\ \bv)|+ |l_1(\bw,\ \bv)| &\le& Ch^t\|\bw\|_{t+1} |\bv|_{1,h},\label{new-mmm1}\\
|s_2(Q_h p,\ q)|+ |l_2(\bw,\ q)| &\le&
Ch^t(\|\bw\|_{t+1}+\|p\|_{t})\ |q|_{0,h},\label{new-mmm888}
\end{eqnarray}
where $l_1(\bw,\bv)$ and $l_2(\bw,q)$ are defined in (\ref{l1}) and
(\ref{l2}).
\end{lemma}

\begin{proof}
Using the definition of $Q_b$, (\ref{trace}) and (\ref{Qh}), we have
\begin{eqnarray*}
|s_1(\bQ_h\bw,\ \bv)|&=&\left|\sum_{T\in\T_h} h^{-1}\langle
(\bQ_0\bw-Q_b\bw)\times\bn,\;
(\bv_0-\bv_b)\times\bn\rangle_\pT\right|\\
&=&\left|\sum_{T\in\T_h} h^{-1} \langle (\bQ_0\bw-\bw)\times\bn,\; (\bv_0-\bv_b)\times\bn\rangle_\pT\right|\\
&\le& \left(\sum_{T\in\T_h}(h^{-2}\|\bQ_0\bw-\bw\|_T^2+\|\nabla (\bQ_0\bw-\bw)\|_T^2)\right)^{1/2} |\bv|_{1,h}\\
&\le& Ch^t\|\bw\|_{t+1} |\bv|_{1,h}.
\end{eqnarray*}
Similarly, we have from (\ref{trace}) and (\ref{Rh}) that
\begin{eqnarray*}
\left|l_1(\bv,\
\bw)\right|&\equiv&\left|\sum_{T\in\T_h}\l(I-\Q_h)\curl\bw,\
\nu(\bv_b-\bv_0)\times\bn
\r_\pT\right|\\
&\le&
\left(\sum_{T\in\T_h}h\|(I-\Q_h)\curl\bw\|_\pT^2\right)^{1/2}|\bv|_{1,h}\\
&\le& Ch^t\|\bw\|_{t+1} |\bv|_{1,h}.
\end{eqnarray*}
This completes the proof of (\ref{new-mmm1}).

As to (\ref{new-mmm888}), note that
\begin{eqnarray*}
|s_2(Q_h p,\ q)|&=&\left|\sum_{T\in\T_h} h\langle Q_0p-Q_b p,\;
q_0-q_b\rangle_\pT\right|\\
&\le &\sum_{T\in\T_h}h |\langle Q_0p- p,\; q_0-q_b\rangle_\pT|\\
&\le& Ch^t\|p\|_{t}\ |q|_{0,h}.
\end{eqnarray*}
It follows from (\ref{trace}) and (\ref{Lh}) that
\begin{eqnarray*}
|l_2(\bw,\ q)|&=&\left|\sum_{T\in\T_h}\langle q_0-q_b,\ (\bw-\bQ_0\bw)\cdot\bn\rangle_\pT\right|\\
&\le& \left(\sum_{T\in\T_h}h^{-1}\|\bw-\bQ_0\bw\|_\pT^2\right)^{1/2}\left(\sum_{T\in\T_h}h\|q_0-q_b\|^2_{\pT}\right)^{1/2}\\
&\le& Ch^t\|\bw\|_{t+1}\ |q|_{0,h}.
\end{eqnarray*}
Combining the above two estimates leads to the inequality
(\ref{new-mmm888}). This completes the proof of the lemma.
\end{proof}

\section{Error Estimates}\label{Section:H1ErrorEstimates}

The objective of this section is to establish some optimal order
error estimates for $\bu_h$ and $p_h$ in certain discrete norms. We
start with a modified {\em inf-sup} condition commonly used for
analyzing saddle point problem.

\smallskip
\begin{lemma}\label{Lemma:inf-sup}
For any $q=\{q_0,q_b\}\in W_h$, there exist a $\bv_q=h^2\{\nabla
q_0, 0\}\in V_{h,0}$ such that
\begin{equation}\label{inf-sup}
b(\bv_q,q)\ge  h^2 \|\nabla q\|_{0,h}^2-C |q|_{0,h}^2
\end{equation}
and
\begin{equation}\label{inf-sup-boundedness}
\3bar \bv_q\3bar \leq C h \|\nabla q\|_{0,h},
\end{equation}
where $C$ is a constant independent of $h$.
\end{lemma}

\begin{proof}
For a given $q=\{q_0,q_b\}\in W_h$ and $\bv=\{\bv_0,\bv_b\}\in
V_{h,0}$, from the definition of the discrete weak gradient we
obtain
\begin{equation*}
\begin{split}
b(\bv,q) &= \sum_{T\in\T_h}(\bv_0, \nabla_w q)_T \\
&= \sum_{T\in\T_h} \left(\l\bv_0\cdot\bn, q_b \r_{\pT}-
(\nabla\cdot\bv_0, q_0)_T\right)\\
&= \sum_{T\in\T_h} \left( (\bv_0, \nabla q_0)_T +\l\bv_0\cdot\bn,
q_b-q_0\r_{\pT}\right),
\end{split}
\end{equation*}
where we have used the usual integration by parts in the last
equation. By choosing $\bv_0=2h^2\nabla q_0$ and $\bv_b=0$ we arrive
at
\begin{equation*}
b(\bv,q) = 2h^2\sum_{T\in\T_h} (\nabla q_0, \nabla q_0)_T
+2h^2\sum_{T\in\T_h} \l\nabla q_0\cdot\bn, q_b-q_0\r_{\pT}.
\end{equation*}
Now by the Cauchy-Schwarz inequality and the trace inequality
(\ref{trace-poly}) we obtain
\begin{equation*}
\begin{split}
b(\bv,q) & \ge  2h^2\sum_{T\in\T_h} (\nabla q_0, \nabla q_0)_T
-2h^2\sum_{T\in\T_h} \|\nabla q_0\cdot\bn\|_{\pT}
\|q_b-q_0\|_{\pT}\\
&\ge 2h^2\sum_{T\in\T_h} (\nabla q_0, \nabla q_0)_T - Ch^{1.5}
\sum_{T\in\T_h} \|\nabla q_0\|_{T} \|q_b-q_0\|_{\pT}\\
&\ge h^2\sum_{T\in\T_h} (\nabla q_0, \nabla q_0)_T - Ch
\sum_{T\in\T_h} \|q_b-q_0\|_{\pT}^2,
\end{split}
\end{equation*}
which gives rise to the inequality (\ref{inf-sup}). The boundedness
estimate (\ref{inf-sup-boundedness}) can be obtain by computing the
triple bar norm of $\bv_q$ directly. This completes the proof of the
lemma.
\end{proof}
\smallskip

The following is an error estimate for the WG finite element
solutions.

\begin{theorem}\label{h1-bd}
Let $(\bu; p)\in [H^{t+1}(\Omega)]^3\times [H_0^1(\Omega)\cap
H^{\max\{1,t\}}(\Omega)]$ with $\frac12< t\le k$ and $(\bu_h;p_h)\in
V_h\times W_h$ be the solution of (\ref{moment1})-(\ref{bc1}) and
(\ref{wg1})-(\ref{wg2}) respectively. Then
\begin{eqnarray}
\3bar\be_h\3bar+|\epsilon_h|_{0,h}&\le&
Ch^t(\|\bu\|_{t+1}+\|p\|_t),\label{err1}\\
h \|\nabla\epsilon_h\|_{0,h}&\leq&Ch^t(\|\bu\|_{t+1}+\|p\|_t).
\label{err1-secondpart}
\end{eqnarray}
\end{theorem}

\smallskip

\begin{proof}
By letting $\bv=\be_h$ in (\ref{ee1}) and $q=\epsilon_h$ in
(\ref{ee2}) and adding the two resulting equations, we have
\begin{eqnarray}
\3bar\be_h\3bar^2+
|\epsilon_h|_{0,h}^2&=&\varphi_\bu(\be_h)+\phi_{\bu,p}(\epsilon_h).
\label{main}
\end{eqnarray}
The right-hand side of (\ref{main}) can be handled by using Lemma
\ref{Lemma:myestimates} as follows. Using (\ref{new-mmm1}) with
$\bw$ and $\bv$ replaced by $\bu$ and $\be_h$ we obtain
\begin{equation}\label{varphi-estimate}
|\varphi_{\bu}(\be_h)|\leq C h^t \|\bu\|_{t+1} \3bar\be_h\3bar.
\end{equation}
Similarly, using (\ref{new-mmm888}) with $\bw$ and $q$ replaced by
$\bu$ and $\epsilon_h$ we obtain
\begin{equation}\label{phi-estimate}
|\phi_{\bu,p}(\epsilon_h)|\leq C h^t (\|\bu\|_{t+1}+\|p\|_t) \
|\epsilon_h|_{0,h}.
\end{equation}
Substituting (\ref{varphi-estimate}) and (\ref{phi-estimate}) into
(\ref{main}) yields
\begin{equation}\label{b-u}
\3bar \be_h\3bar^2+|\epsilon_h|_{0,h}^2 \le
Ch^t(\|\bu\|_{t+1}+\|p\|_t)(\3bar \be_h\3bar+ |\epsilon_h|_{0,h}),
\end{equation}
which implies the error estimate (\ref{err1}).

Next we will bound $\|\nabla\epsilon_h\|_{0,h}$. It follows from
(\ref{ee1}) that
\[
b(\bv,\ \epsilon_h)=a(\be_h,\ \bv)-\varphi_\bu(\bv)\qquad \forall
\bv\in V_{h,0}.
\]
From Lemma \ref{Lemma:inf-sup}, by choosing
$\bv=\bv_{\epsilon_h}=h^2\{\nabla \epsilon_h, 0\}$ we come up with
\begin{equation}\label{raining.100}
\begin{split}
h^2\|\nabla\epsilon_h\|_{0,h}^2&\leq |b(\bv_{\epsilon_h},
\epsilon_h)| + C |\epsilon_h|_{0,h}^2\\
&\leq |a(\be_h,\ \bv_{\epsilon_h})| +
|\varphi_\bu(\bv_{\epsilon_h})|+C |\epsilon_h|_{0,h}^2\\
&\leq \3bar \be_h\3bar\ \3bar \bv_{\epsilon_h}\3bar +
Ch^t\|\bu\|_{t+1} \3bar\bv_{\epsilon_h}\3bar + C
 |\epsilon_h|_{0,h}^2\\
&\leq C(\3bar \be_h\3bar + Ch^t\|\bu\|_{t+1}) h
\|\nabla\epsilon_h\|_{0,h} + C |\epsilon_h|_{0,h}^2,
\end{split}
\end{equation}
where we have used the estimate (\ref{inf-sup-boundedness}) in the
last inequality. It follows from (\ref{raining.100}) and
(\ref{err1}) that (\ref{err1-secondpart}) holds true. This completes
the proof of the theorem.
\end{proof}

\medskip
Recall that $\3bar\bv\3bar$ is merely a semi-norm in the finite
element space $V_{h,0}$. Thus, the error estimate (\ref{err1}) only
provides a partial answer to the convergence of the WG finite
element method, particularly for the vector component $\bu_h$. The
norm $\3bar\cdot\3bar_1$, as defined by (\ref{3bar1nomr}), involves
two additional terms. The following theorem shall provide some
estimates with respect to those additional terms.

\medskip

\begin{theorem}\label{Theorem:divpart}
Let $(\bu; p)\in [H^{t+1}(\Omega)]^3\times (H_0^1(\Omega)\cap
H^{\max\{1,t\}}(\Omega))$ with $\frac12< t\le k$ and $(\bu_h;p_h)\in
V_h\times W_h$ be the solution of (\ref{moment1})-(\ref{bc1}) and
(\ref{wg1})-(\ref{wg2}) respectively. Then, we have
\begin{eqnarray}
\left(\sum_{e\in\E_h^0}h^{-1}\|\jump{\be_0\cdot\bn}\|_e^2
\right)^{\frac12} &\leq& Ch^t(\|\bu\|_{t+1}+\|p\|_t),\label{t-1}\\
\left(\sum_{T\in\T_h} \|\nabla\cdot\be_0\|_T^2 \right)^{\frac12} &\leq&
Ch^t(\|\bu\|_{t+1}+\|p\|_t).\label{t-2}
\end{eqnarray}
\end{theorem}

\begin{proof}
Using the error equation (\ref{ee2}) we have
\begin{equation}\label{normal}
b(\be_h,\ q)=\phi_{\bu,p}(q)-s_2(\epsilon_h,\ q).
\end{equation}
The definition of the weak gradient implies that
\begin{eqnarray}
b(\be_h,q) &=&  (\be_0, \nabla_w q)
=\sum_{T\in\T_h} \left(\l\be_0\cdot\bn,\ q_b\r_\pT - (\nabla\cdot\be_0,\ q_0
)_T \right).\label{t0}
\end{eqnarray}
By letting $q=q_{\be_h}=\{0,h^{-1}\jump{\be_0\cdot\bn}\}$ on the
interior edges, we obtain
\begin{equation*}
b(\be_h,q_{\be_h})=\sum_{e\in\E_h^0}h^{-1}\|\jump{\be_0\cdot\bn}\|_e^2.
\end{equation*}
Thus,
\begin{equation*}\label{t1}
\sum_{e\in\E_h^0}h^{-1}\|\jump{\be_0\cdot\bn}\|_e^2=\phi_{\bu,p}(q_{\be_h})-s_2(\epsilon_h,\
q_{\be_h}).
\end{equation*}
It follows from (\ref{err1}) that
\begin{equation}\label{t3}
|s_2(\epsilon_h,\ q_{\be_h})|\le | \epsilon_h|_{0,h}\ |
q_{\be_h}|_{0,h}\le Ch^t(\|\bu\|_{t+1}+\|p\|_t)\ |q_{\be_h}|_{0,h},
\end{equation}
and from (\ref{new-mmm888})
\begin{equation}\label{t3-new}
|\phi_{\bu,p}(q_{\be_h})|\le Ch^t(\|\bu\|_{t+1}+\|p\|_t)\ |
q_{\be_h}|_{0,h}.
\end{equation}
Also, it is easy to see that
\begin{equation*}
\begin{split}
|q_{\be_h}|_{0,h}^2&=\sum_{T\in\T_h}h\|q_0-q_b\|_{\pT\cap\Omega}^2\\
&=\sum_{T\in\T_h}h^{-1}\|\jump{\be_0\cdot\bn}\|_{\pT\cap\Omega}^2\\
&\le C\sum_{e\in\E_h^0}h^{-1}\|\jump{\be_0\cdot\bn}\|_e^2.
\end{split}
\end{equation*}
Combining the above four inequalities yields
\begin{equation}\label{t100}
\left(\sum_{e\in\E_h^0}h^{-1}\|\jump{\be_0\cdot\bn}\|_e^2\right)^{\frac12}
\leq C h^t(\|\bu\|_{t+1}+\|p\|_t),
\end{equation}
which verifies the estimate (\ref{t-1}).

To derive (\ref{t-2}), we set $q=q_{\be_h}=\{-\nabla\cdot\be_0,\
0\}\in W_h$ in (\ref{t0}) so that
\begin{equation}\label{t6}
b(\be_h,q_{\be_h})=\sum_{T\in\T_h}\|\nabla\cdot\be_0\|_T^2.
\end{equation}
Thus, it follows from (\ref{normal}) that
\begin{equation}\label{t888}
\sum_{T\in\T_h}\|\nabla\cdot\be_0\|_T^2=\phi_{\bu,p}(q_{\be_h})-s_2(\epsilon_h,\
q_{\be_h}).
\end{equation}
Substituting (\ref{t3}) and (\ref{t3-new}) into (\ref{t888}) implies
\begin{equation}\label{t999}
\sum_{T\in\T_h}\|\nabla\cdot\be_0\|_T^2\le
Ch^t(\|\bu\|_{t+1}+\|p\|_t)\ |q_{\be_h}|_{0,h}.
\end{equation}
It follows from the definition of $|q|_{0,h}$ and the trace
inequality (\ref{trace-poly}) that
$$
|q_{\be_h}|_{0,h} \leq \left(\sum_{T\in\T_h}
h\|\nabla\cdot\be_0\|_\pT^2 \right)^{\frac12}\leq
C\left(\sum_{T\in\T_h} \|\nabla\cdot\be_0\|_T^2 \right)^{\frac12},
$$
which, together with (\ref{t999}), leads to the following estimate
$$
\left(\sum_{T\in\T_h} \|\nabla\cdot\be_0\|_T^2 \right)^{\frac12} \leq
Ch^t(\|\bu\|_{t+1}+\|p\|_t).
$$
This completes the proof.
\end{proof}

\medskip
To summarize, we have obtained the following error estimate for the
WG finite element solution arising from (\ref{wg1})-(\ref{wg2}).
\begin{theorem}\label{WG-ErrorEstimate}
Under the assumptions of Theorem \ref{h1-bd}, we have the following
error estimate for the WG finite element approximations:
\begin{eqnarray}
\3bar\be_h\3bar_1+\3bar\epsilon_h\3bar_{0}&\le&
Ch^t(\|\bu\|_{t+1}+\|p\|_t).\label{HONEerr}
\end{eqnarray}
\end{theorem}

\section{An Error Estimate in $L^2$}\label{Section:L2ErrorEstimates}

To derive an $L^2$-error estimate for the WG approximation of the
vector component, we consider an auxiliary problem that seeks
$(\psi;\xi)$ satisfying
\begin{equation}\label{dual-problem}
\begin{split}
\curl(\nu \curl\psi)-\nabla \xi &=\be_0\quad \mbox{in}\;\Omega,\\
\nabla\cdot\psi&=0\quad\mbox{in}\;\Omega,\\
\psi\times\bn &= 0\quad\mbox{on}\;\partial\Omega,\\
\xi &= 0\quad\mbox{on}\;\partial\Omega.\\
\end{split}
\end{equation}
Assume that the problem (\ref{dual-problem}) has the
$[H^{1+s}(\Omega)]^3\times H^s(\Omega)$-regularity property in the
sense that the solution $(\psi; \xi)\in (H^{1+s}(\Omega))^3\times
H^s(\Omega)$ and the following a priori estimate holds true:
\begin{equation}\label{reg}
\|\psi\|_{1+s}+\|\xi\|_s\le C\|\be_0\|,
\end{equation}
where $0< s\le 1$.

\medskip
\begin{theorem} \label{L2-bd}
Let $(\bu; p)\in  [H^{r+1}(\Omega)]^3\times [H_0^1(\Omega)\cap
H^{\max\{1,r\}}(\Omega)]$ and $(\bu_h;p_h)\in V_h\times W_h$ be the
solutions of (\ref{moment1})-(\ref{bc1}) and (\ref{wg1})-(\ref{wg2})
respectively. Let $\frac12 <r\le k$ and $0 < s \le 1$. Then,
\begin{equation}\label{l2-error}
\|\bQ_0\bu-\bu_0\|\le Ch^{r+s}(\|\bu\|_{r+1}+\|p\|_r).
\end{equation}
\end{theorem}

\begin{proof}
Testing the first equation of (\ref{dual-problem}) by $\be_0$ gives
\[
\|\bQ_0\bu-\bu_0\|^2=(\be_0,\ \be_0)=(\curl(\nu\curl\psi),\
\be_0)-(\nabla\xi,\ \be_0).
\]
Using (\ref{m2}) and (\ref{m4}) (with $\psi, \xi, \be_h$ in the
place of $\bu, p, \bv$ respectively), the above equation becomes
\begin{eqnarray*}
\|\bQ_0\bu-\bu_0\|^2&=&(\nu\cw\bQ_h\psi,\ \cw\be_h)_h-(\be_0,\
\nabla_w(Q_h\xi))_h+l_1(\psi,\ \be_h).
\end{eqnarray*}
Adding and subtracting $s_1(\bQ_h\psi,\ \be_h)$ to the equation
above yields
\begin{eqnarray*}
\|\bQ_0\bu-\bu_0\|^2&=&a(\bQ_h \psi,\ \be_h)-b(\be_h,\
Q_h\xi)-\varphi_\psi(\be_h).
\end{eqnarray*}
The error equation (\ref{ee2}) implies
\[
b(\be_h,\ Q_h\xi)=-s_2(\epsilon_h,\ Q_h\xi)+\phi_{\bu,p}(Q_h\xi).
\]
It now follows from the definition of $\bQ_0$, $\nabla_w$ and the
second equation of (\ref{dual-problem}) that
\begin{eqnarray*}
b(\bQ_h\psi,\ \epsilon_h)&=&(\bQ_0\psi, \nabla_w\epsilon_h)_h=\sum_{T\in\T_h}(\l \epsilon_b,\ \bQ_0\psi\cdot\bn\r_\pT-(\epsilon_0,\ \nabla\cdot(\bQ_0\psi))_T)\\
&=&\sum_{T\in\T_h}(\l \epsilon_b-\epsilon_0,\ \bQ_0\psi\cdot\bn\r_\pT+(\nabla\epsilon_0,\ \psi)_T)\\
&=&\sum_{T\in\T_h}(\l \epsilon_b-\epsilon_0,\ \bQ_0\psi\cdot\bn\r_\pT-\l \epsilon_b-\epsilon_0,\ \psi\cdot\bn\r_\pT)\\
&=&l_2(\psi,\epsilon_h),
\end{eqnarray*}
where we have used the fact that $\sum_{T\in\T_h}\l \epsilon_b,\ \psi\cdot\bn\r_\pT=0$.
Using the equations above, we have
\begin{eqnarray*}
\|\bQ_0\bu-\bu_0\|^2&=&a(\bQ_h\psi,\ \be_h)-b(\bQ_h\psi,\
\epsilon_h)
-\phi_{\bu,p}(Q_h\xi)-\varphi_\psi(\be_h)+\phi_{\psi,\xi}(\epsilon_h).
\end{eqnarray*}
Using (\ref{ee1}) and the equation above, we have
\begin{eqnarray} \label{d1}
\|\bQ_0\bu-\bu_0\|^2=\varphi_\bu(\bQ_h\psi)-\phi_{\bu,p}(Q_h\xi)-\varphi_\psi(\be_h)+\phi_{\psi,\xi}(\epsilon_h).
\end{eqnarray}

The four terms on the right-hand side of (\ref{d1}) can be handled
by the estimates presented in Lemma \ref{Lemma:myestimates}. To this
end, we use (\ref{new-mmm1}) and (\ref{new-mmm888}) with $t=r$ to
obtain
\begin{equation}\label{EQN:09-21-001}
|\varphi_\bu(\bQ_h\psi)-\phi_{\bu,p}(Q_h\xi)|\leq C
h^r(\|\bu\|_{r+1}+\|p\|_r) \left( |\bQ_h\psi|_{1,h} + |Q_h\xi|_{0,h}
\right).
\end{equation}
Using the definition (\ref{Semi-Norm-v1h}) we have
\begin{equation} \label{d3}
\begin{split}
|\bQ_h\psi|_{1,h}^2=&\sum_{T\in\T_h} h^{-1}\|(\bQ_0\psi-Q_b\psi)\times\bn\|^2_\pT\\
\le& \sum_{T\in\T_h} h^{-1}\|(\bQ_0\psi-\psi)\times\bn\|_\pT^2\\
\le& Ch^{2s}\|\psi\|_{s+1}^2.\end{split}
\end{equation}
Similarly, we have from the definition of $Q_b$, (\ref{trace}) and
(\ref{Lh})
\begin{equation}\label{d4}
\begin{split}
|Q_h\xi|_{0,h}^2&=\sum_{T\in\T_h} h\|Q_0\xi-Q_b\xi\|^2_\pT\\
&\le \sum_{T\in\T_h} h\|Q_0\xi-\xi\|_\pT^2\\
&\le Ch^{2s}\|\xi\|_{s}^2.
\end{split}
\end{equation}
Substituting (\ref{d3}) and (\ref{d4}) into (\ref{EQN:09-21-001})
gives
\begin{equation}\label{EQN:09-21-002}
\begin{split}
|\varphi_\bu(\bQ_h\psi)-\phi_{\bu,p}(Q_h\xi)|&\leq C
h^{r+s}(\|\bu\|_{r+1}+\|p\|_r) \left( \|\psi\|_{1+s} + \|\xi\|_s
\right)\\
&\leq C h^{r+s}(\|\bu\|_{r+1}+\|p\|_r) \|\be_0\|,
\end{split}
\end{equation}
where the regularity estimate (\ref{reg}) was used in the second
equation.

Analogously, we have from (\ref{new-mmm1}) and (\ref{new-mmm888})
with $t=s$ that
\begin{equation}\label{EQN:09-21-005}
\begin{split}
|\varphi_\psi(\be_h)-\phi_{\psi,\xi}(\epsilon_h)| &\leq C
h^s(\|\psi\|_{s+1}+\|\xi\|_s) \left( |\be_h|_{1,h} + |
\epsilon_h|_{0,h} \right)\\
&\leq C h^s \left( \3bar\be_h\3bar + | \epsilon_h|_{0,h} \right)
\|\be_0\|\\
&\le C h^{r+s}(\|\bu\|_{r+1}+\|p\|_r) \|\be_0\|,
\end{split}
\end{equation}
where we have used the error estimate (\ref{err1}) and the
regularity inequality (\ref{reg}). Finally, substituting
(\ref{EQN:09-21-002}) and (\ref{EQN:09-21-005}) into (\ref{d1})
yields the desired error estimate (\ref{l2-error}). This completes
the proof of the theorem.
\end{proof}

\section{An Effective Implementation through Variable Reduction}\label{Section:Schur}
The degree of freedoms for the WG formulation
(\ref{wg1})-(\ref{wg2}) is associated with $\bu_h=\{\bu_0,\bu_b\}$
and $p_h=\{p_0,p_b\}$. In this section, we will demonstrate how
$\bu_0$ and $p_0$ can be eliminated from the system in order to
obtain a global system that depends only on $\bu_b$ and $p_b$. With
such a variable reduction, the number of unknowns of the WG method
is reduced significantly for an efficient practical implementation.

Let $\bu_h=\{\bu_0,\bu_b\}\in V_h$ and $p_h=\{p_0,p_b\}\in W_h$ be
the solution of the WG method (\ref{wg1})-(\ref{wg2}). Recall that
$(\bu_h;p_h)$ satisfies $\bu_b\times\bn=Q_b\phi$ on $\partial\Omega$
and the following equations:
\begin{eqnarray}
a(\bu_h,\ \bv)-b(\bv,\;p_h)&=&({\bf f},\;\bv_0),\quad\forall\ \bv=\{\bv_0,\; 0\}\in
V_{h,0},\label{wg11}\\
b(\bu_h,\;q)+s_2(p_h,q)&=&-(g,q_0), \quad\forall\ q=\{q_0,\; 0\}\in
W_h,\label{wg22}
\end{eqnarray}
and
\begin{eqnarray}
a(\bu_h,\ \bv)&=&0,\quad\forall\ \bv=\{0,\; \bv_b\}\in
V_{h,0},\label{wg111}\\
b(\bu_h,\;q)+s_2(p_h,q)&=&0, \quad\forall\ q=\{0,\; q_b\}\in
W_h.\label{wg222}
\end{eqnarray}
Denote by $V_k(T)$ and $W_k(T)$ the  restrictions of $V_h$ and $W_h$
on $T$:
\[
V_k(T)=\{ \bv=\{\bv_0, \bv_b=v_1\bt_1+v_2\bt_2\}  :\ \bv_0|_{T}\in
      [P_k(T)]^3,\  v_1, v_2\in  P_k(e),\ e\subset\pT\}.
\]
and
\[
W_k(T)=\{q=\{q_0,q_b\}, q_0\in P_{k-1}(T), q_b\in
P_{k}(e),e\subset\pT\}.
\]
Since the testing functions $\bv=\{\bv_0,\; 0\}$ and $q=\{q_0,\;
0\}$ are locally supported on each element, then the system of
equations (\ref{wg11})-(\ref{wg22}) is equivalent to the following
system of equations defined locally on each element $T$:
\begin{eqnarray}
a(\bu_h,\ \bv)-b(\bv,\;p_h)&=&({\bf f},\;\bv_0),\quad\forall\ \bv=\{\bv_0,\; 0\}\in
V_k(T),\label{wg1-local}\\
b(\bu_h,\;q)+s_2(p_h,q)&=&-(g,q_0), \quad\forall\ q=\{q_0,\; 0\}\in
W_k(T).\label{wg2-local}
\end{eqnarray}
If the exact solution of $\bu_b$ and $p_b$ were known on $\pT$, then
the corresponding $\bu_0$ and $p_0$ can be obtained by solving
(\ref{wg1-local}) and (\ref{wg2-local}) locally on each element.
Therefore, the key is to derive a system of equations that shall
determine $\bu_b$ and $p_b$.

For any given $\bu_b$ and $p_b$ on $\pT$, let us solve
(\ref{wg1-local}) and (\ref{wg2-local}) to obtain $\bu_0$ and $p_0$
on each element $T$. For simplicity, we introduce the following
notation
\begin{eqnarray}
\bu_0:&=& D(\bu_b,p_b,{\bf f},g),\label{EQN:09-21-200}\\
p_0:&=& E(\bu_b, p_b,f,g).\label{EQN:09-21-201}
\end{eqnarray}
Then the solution $\bu_h$ and $p_h$ of (\ref{wg1})-(\ref{wg2}) can
be written as $\bu_h=\{\bu_0,\bu_b\}=\{D(\bu_b,p_b,{\bf
f},g),\bu_b\}$ and $p_h=\{p_0,p_b\}=\{E(\bu_b,p_b,{\bf f},g),p_b\}$.

Let $D_1(\bu_b,p_b)=D(\bu_b,p_b,0,0)$ \ and \ $D_2({\bf
f},g)=D(0,0,{\bf f},g)$. Similarly let $E_1(\bu_b,p_b)  =
E(\bu_b,p_b,0,0)$ and $E_2({\bf f},g)=E(0,0,{\bf f},g)$. Since
$a(\cdot,\cdot)$, $b(\cdot,\cdot)$ and $s_2(\cdot,\cdot)$ are
bilinear, then superposition implies
\begin{equation*}
\begin{split}
(\bu_h; p_h)&=(\{\bu_0,\bu_b\}; \{p_0,p_b\})\\
&=(\{ D(\bu_b,p_b,\bbf,g), \bu_b\}; \{E(\bu_b,p_b,\bbf,g),p_b\})\\
&=(\{ D(\bu_b,p_b,0,0), \bu_b\}; \{E(\bu_b,p_b,0,0),p_b\})\\
&\ \ + (\{ D(0,0,\bbf,g), 0\}; \{E(0,0,\bbf,g),0\})\\
&=(\{ D_1(\bu_b,p_b), \bu_b\}; \{E_1(\bu_b,p_b),p_b\}) + (\{
D_2(\bbf,g), 0\}; \{E_2(\bbf,g),0\}).
\end{split}
\end{equation*}

Substituting $\bu_h=\{D(\bu_b,p_b,{\bf f},g),\bu_b\}$ and
$p_h=\{E(\bu_b,p_b,{\bf f},g),p_b\}$ into the system
(\ref{wg111})-(\ref{wg222}) yields
\begin{eqnarray}
a(\{D_1(\bu_b,p_b),\bu_b\},\ \bv)&=&\zeta_1(\bv),\label{wg1111}\\
b(\{D_1(\bu_b,p_b),\bu_b\},\;q)+s_2(\{E_1(\bu_b,p_b),p_b\},q)&=&\zeta_2(q),
\label{wg2222}
\end{eqnarray}
for all $\bv=\{0,\; \bv_b\}\in V_{h,0}$ and $q=\{0,\; q_b\}\in W_h$.
Here
\begin{eqnarray*}
\zeta_1(\bv)&=&-a(\{D_2({\bf f},g),0\},\bv)\\
\zeta_2(q)&=&-b(\{D_2({\bf f},g),0\},q)-s_2(\{E_2({\bf f},g),0\},q).
\end{eqnarray*}
Note that the system (\ref{wg1111})-(\ref{wg2222}) is a square
matrix problem with $\bu_b$ and $p_b$ as unknowns, and this is the
system of equations that $\bu_b$ and $p_b$ have to satisfy.

To summarize, our WG scheme (\ref{wg1})-(\ref{wg2}) can be
implemented as follows:

\begin{description}
\item[{\bf Step 1.}] Find $\bu_b$ and $p_b$ with
$\bu_b\times\bn=Q_b\phi$ and $p_b=0$ on $\partial\Omega$ satisfying
(\ref{wg1111})-(\ref{wg2222}).
\item[{\bf Step 2.}] Recover $\bu_0$
and $p_0$ by $\bu_0=D(\bu_b,p_b,{\bf f},g)$ and
$p_0=E(\bu_b,p_b,{\bf f},g)$ by solving (\ref{wg1-local}) and
(\ref{wg2-local}) locally on each element.
\end{description}

\smallskip
The system of equations (\ref{wg1111})-(\ref{wg2222}) is known as a
Schur complement of the original WG finite element scheme
(\ref{wg1})-(\ref{wg2}).

\section{Numerical Results}\label{Section:NE}
Our numerical tests are conducted for the Maxwell equations
\eqref{moment1}--\eqref{bc1} on the unit cube $\Omega=(0,1)^3$. The
first level grid consists of one cube. Each grid is refined by
subdividing a cube into eight half-sized cubes, to define the next
level grid. We apply the first order weak Galerkin finite element
method; i.e., $V_h$ and $W_h$ are defined in \eqref{Vh} and
\eqref{Wh} with $k=1$, respectively. Thus, the vector component
$\bu$ is approximated by using piecewise linear functions on each
cube and its faces; the scalar component $p$ is approximated by
using constants on each cube and linear function on its faces.

We compute four sets of solutions of \eqref{moment1}--\eqref{bc1},
which are
\begin{align}
    \label{s1} && \bu&= \begin{pmatrix} y-z \\ z-x \\ 3z-2y \end{pmatrix},
     &  p&=1.  && \\
   \label{s2} && \bu&= \begin{pmatrix} yz \\ zx \\ 3z-2yx \end{pmatrix},
     &  p&=xz.  && \\
   \label{s3} && \bu&= \begin{pmatrix} e^{yz} \\ z/(x+1) \\ e^{xy} \end{pmatrix},
     &  p&=e^{-xyz}. &&\\
\label{s4} && \bu&= \begin{pmatrix} \cos(\pi x)\sin(\pi y)\sin(\pi z) \\
                              \sin(\pi x)\cos(\pi y)\sin(\pi z) \\
                              \sin(\pi x)\sin(\pi y)\cos(\pi z) \end{pmatrix},
     &  p&=\sin(2\pi x)\sin(2\pi y)\sin(2\pi z).  &&
\end{align}


\begin{figure}[htb]\setlength\unitlength{1in}
\centering
    \begin{picture}(3,5.0)(0,0)
  \put(0,1.5){\resizebox{3.5in}{4.5in}{\includegraphics{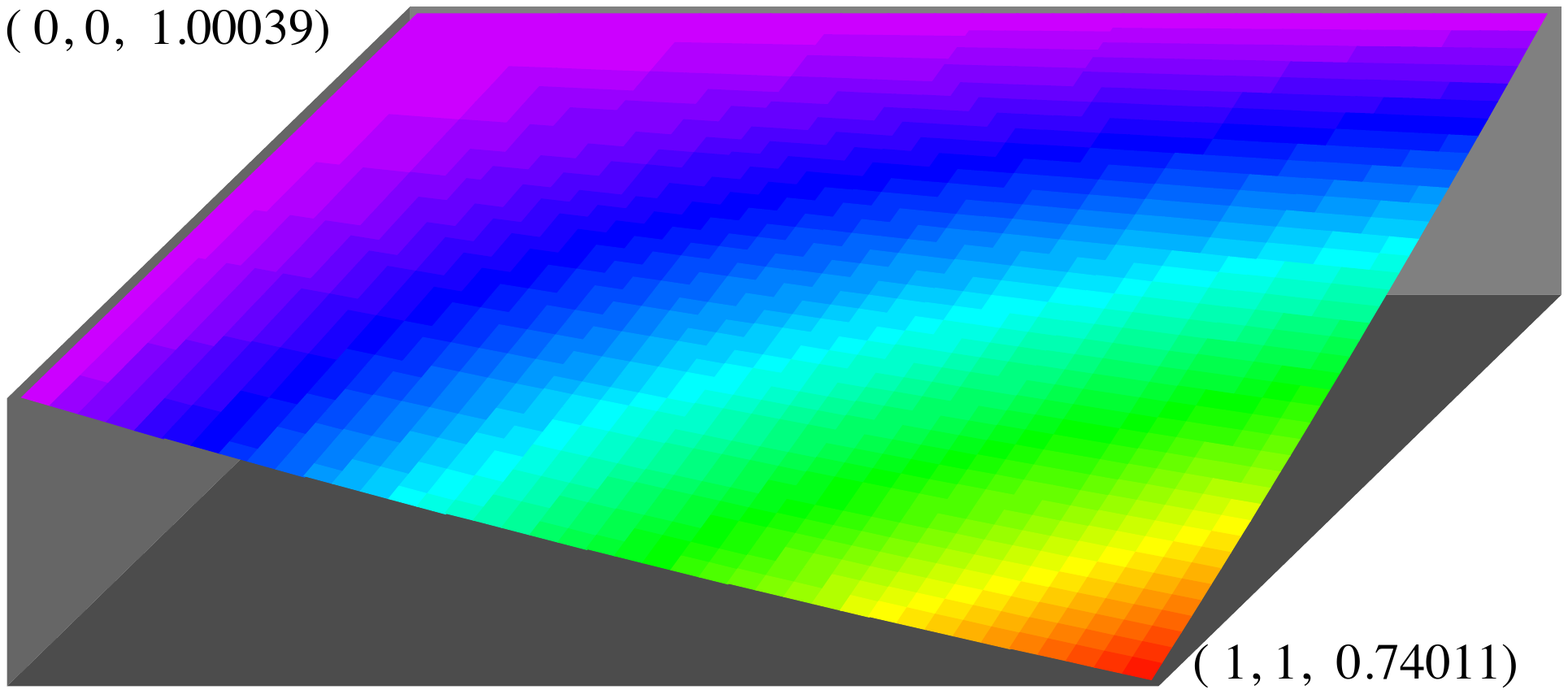}}}
   \put(0,-0.5){\resizebox{3.5in}{4.5in}{\includegraphics{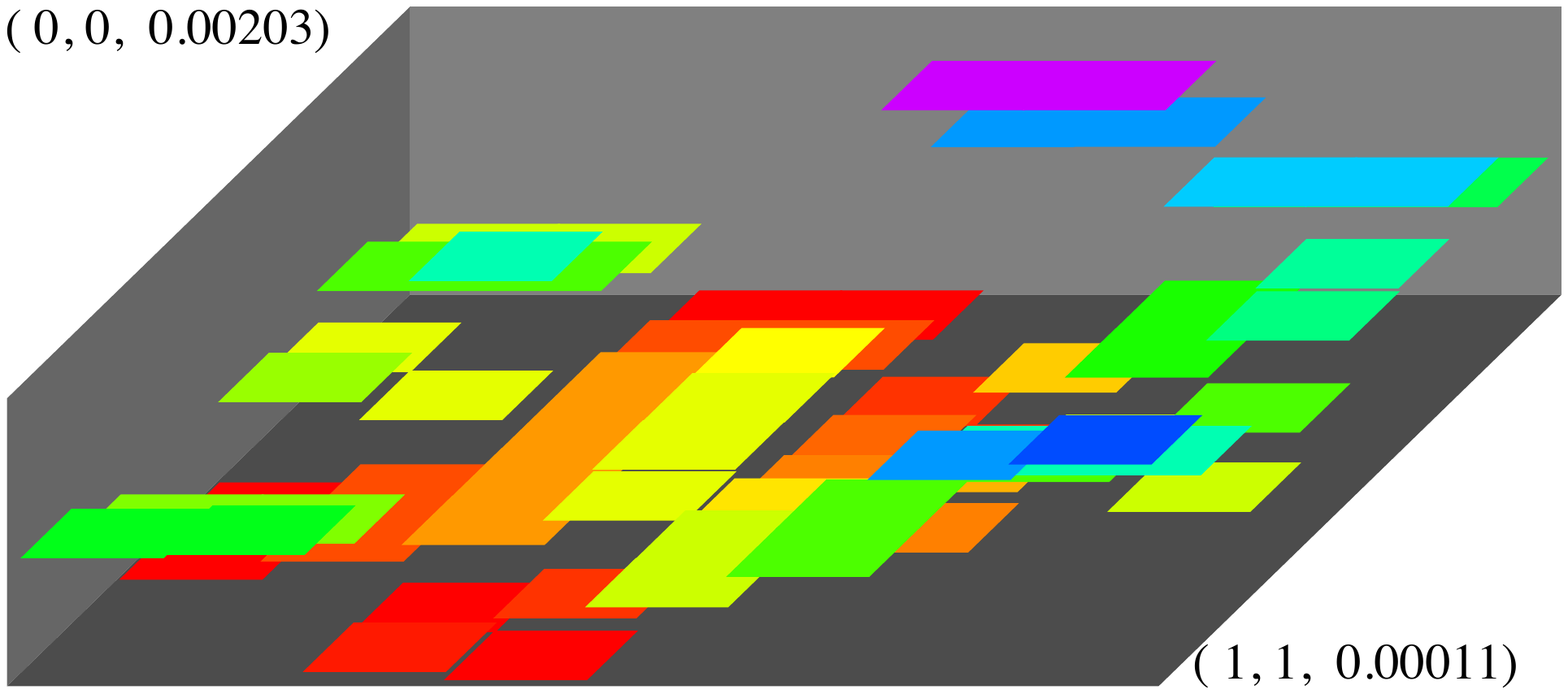}}}
   \put(0,-2.5){\resizebox{3.5in}{4.5in}{\includegraphics{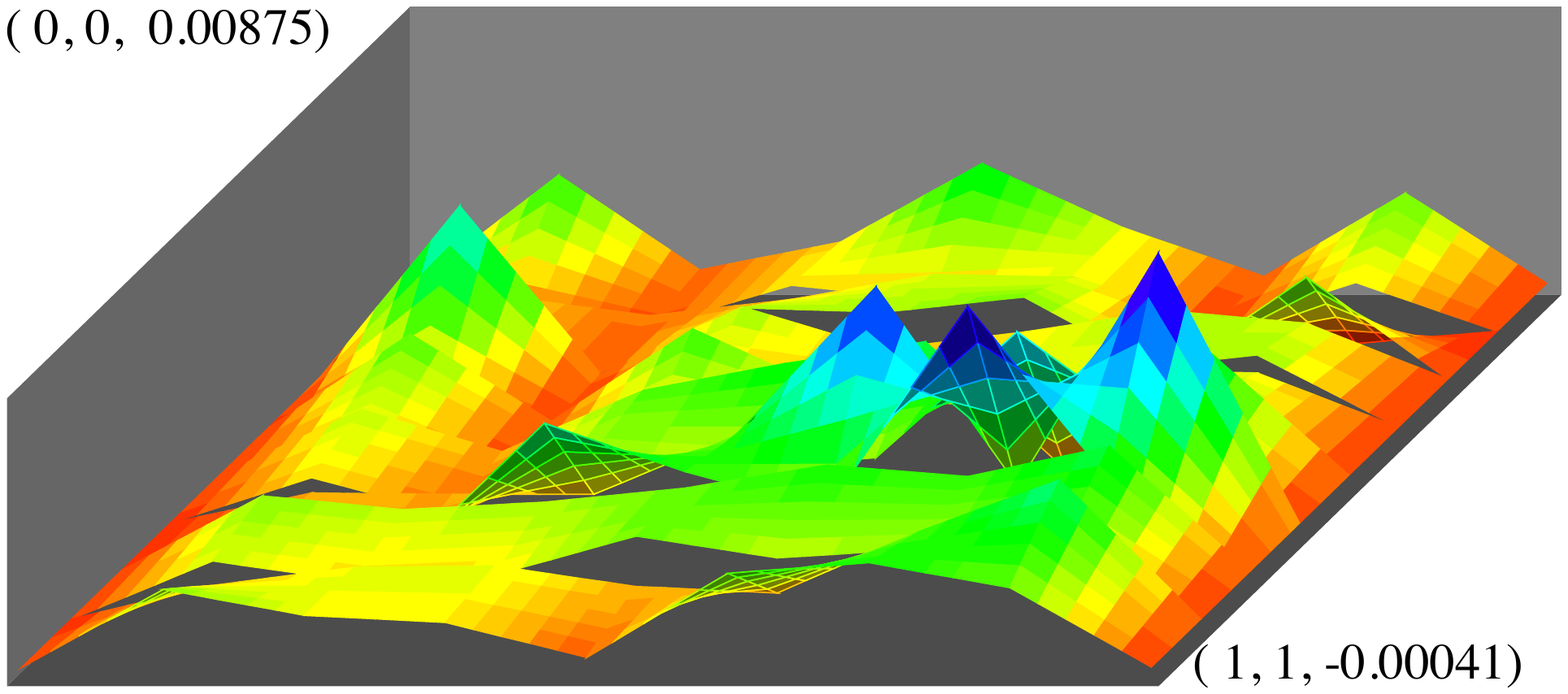}} }
    \end{picture}
    \caption{The solution $p$ in \eqref{s3}, and the errors
$(p-p_0)$ and $(p-p_b)$ on level 4, at $z=0.3$. } \label{f-p}
\end{figure}

Observe that the solution $p$ in the first three test cases does not
satisfy the homogeneous boundary condition (\ref{bc1}). The
corresponding WG scheme (\ref{wg1})-(\ref{wg2}) needs to be modified
so that $p_b$ assumes the given non-homogeneous boundary value;
namely, the $L^2$ projection of the value of $p$ on the boundary.

 \begin{table}[htb]
  \caption{ \label{table1} The errors,  $\be_h=\bQ_h \bu- \bu_h$ in $H^1$-like norm $\3bar\cdot\3bar_1$,
      $\epsilon_h=Q_hp -p_h$ in $L^2$-like norm $\3bar\cdot\3bar_0$,
     $\be_0=\bQ_0 \bu - \bu_0$ in $L^2$ norm,
    and $\epsilon_0=Q_0p-p_0$ in $L^2$ norm,
    for \eqref{s1} by $k=1$ finite elements \eqref{Vh}--\eqref{Wh}.}
\begin{center}  \begin{tabular}{c|rr|rr|rr|rr}  
 \hline grid &  $ \3bar \be_h\3bar_1$&$h^r$ &  $ \3bar\epsilon_h \3bar_0$&$h^r$ &
   $ \|\epsilon_0\|_{L^2}$&$h^r$  & $ \|\be_0\|_{L^2}$&$h^r$ \\ \hline
1&     0.00000&0.0&     0.00000&0.0&   0.00000&0.0&     0.00000&0.0  \\
2&     0.00000&0.0&     0.00000&0.0&    0.00000&0.0&    0.00000&0.0  \\
3&     0.00000&0.0&     0.00000&0.0&    0.00000&0.0&    0.00000&0.0  \\
4&     0.00000&0.0&     0.00000&0.0&    0.00000&0.0&    0.00000&0.0  \\
   \hline
\end{tabular}\end{center} \end{table}

The first numerical test is used to check the correctness of the
code, where the exact solutions \eqref{s1} are linear and constant,
respectively. As expected, the weak Galerkin finite element
solutions are exact, up to computer accuracy. As shown in Table
\ref{table1}, all errors are zero.


\begin{table}[htb]
  \caption{ \label{table2} The errors, $\be_h=\bQ_h \bu- \bu_h$ in $H^1$-like norm
  $\3bar \be_h\3bar_1$,
  $\be_0=\bQ_0 \bu - \bu_0$ in $L^2$ norm $\|\be_0\|$,
      $\epsilon_h=Q_hp -p_h$ in $L^2$-like norm $\3bar\epsilon_h\3bar_0$,
    and $\epsilon_0=Q_0p-p_0$ in $L^2$ norm $\|\epsilon_0\|$,
    for \eqref{s2} by $k=1$ finite elements \eqref{Vh}--\eqref{Wh}.
   And the order $r$ as in $O(h^r)$ of convergence.}
\begin{center}  \begin{tabular}{c|lc|lc|lc|lc}  
\hline grid &  $ \3bar\be_h\3bar_1$&$h^r$ & $ \|\be_0\|_{L^2}$&$h^r$
& $ \3bar\epsilon_h\3bar_0$&$h^r$ &$ \|\epsilon_0\|_{L^2}$&$h^r$
\\ \hline
   1   &2.26e-08   &-     &7.76e-09   &-     &0.0000   &-   &0.0000  &-\\
   2   &5.15e-02   &-     &9.46e-03   &-     &0.0000   &-   &0.0000  &-\\
   3   &2.28e-02   &1.1   &2.14e-03   &2.1   &0.0000   &-   &0.0000  &-\\
   4   &8.77e-03   &1.4   &4.15e-04   &2.4   &0.0000   &-   &0.0000  &-\\
   5   &3.03e-03   &1.5   &7.66e-05   &2.4   &0.0000   &-   &0.0000  &-\\
   \hline
\end{tabular}\end{center} \end{table}

In the second test \eqref{s2}, we choose some bilinear functions as
the exact solution. The numerical results are reported in Table
\ref{table2}. It can be seen that the numerical solution for the
unknown function $p$ is numerically the same as the exact solution.
Moreover, the order of convergence for $\bu$ is half-order higher
than what was proved in the theory. This superconvergence is
probably caused by the special format of the exact solution.

\begin{table}[htb]
  \caption{ \label{table3} The errors,  $\be_h=\bQ_h \bu- \bu_h$ in $H^1$-like norm
  $\3bar \be_h\3bar_1$, $\be_0=\bQ_0 \bu - \bu_0$ in $L^2$ norm $\|\be_0\|$,
      $\epsilon_h=Q_hp -p_h$ in $L^2$-like norm $\3bar\epsilon_h\3bar_0$,
$\3bar\epsilon_h\3bar_{0,h}$, and $\epsilon_0=Q_0p-p_0$ in $L^2$
norm $\|\epsilon_0\|$,
    for \eqref{s3} by $k=1$ finite elements \eqref{Vh}--\eqref{Wh}.
   And the order $r$ as in $O(h^r)$ of convergence.}
\begin{center}  \begin{tabular}{c|lc|lc|lc|lc|lc}  
\hline grid &  $ \3bar \be_h\3bar_1$&$h^r$ & $
\|\be_0\|_{L^2}$&$h^r$ &  $ \3bar\epsilon_h \3bar_0$&$h^r$ & $
\3bar\epsilon_h \3bar_{0,h}$&$h^r$ &$ \|\epsilon_0\|_{L^2}$&$h^r$   \\
\hline
1   &7.02e-1  &-     &3.32e-1   & -    &6.56e-3   &-     &6.56e-3 &-    &2.68e-3  &-\\
2   &3.69e-1  &0.9   &8.71e-2   &1.9   &7.34e-2   &-     &4.73e-3 &0.5  &2.33e-3  &0.2\\
3   &1.91e-1  &0.9   &2.10e-2   &2.1   &5.11e-2   &0.5   &1.09e-3 &2.1  &4.73e-4  &2.3\\
4   &1.02e-1  &0.9   &5.10e-3   &2.0   &2.91e-2   &0.8   &2.67e-4 &2.0  &1.18e-4  &2.0\\
5   &5.05e-2  &1.0   &1.26e-3   &2.0   &1.55e-2   &0.9   &6.59e-5 &2.0  &2.95e-5  &2.0\\
6   &2.52e-2  &1.0   &3.16e-4   &2.0   &7.73e-3   &1.0   &1.65e-5 &2.0  &7.39e-6  &2.0\\
\hline
\end{tabular}\end{center} \end{table}

In the third test \eqref{s3}, the exact solution is chosen as a
general function. The numerical results for this test case is
presented in Table \ref{table3},  confirming the theoretical
convergence estimates as derived in Theorems \ref{WG-ErrorEstimate}
and \ref{L2-bd}.

\begin{figure}[htb]\setlength\unitlength{1in}
\centering
    \begin{picture}(3,6.7)(0,0)
  \put(0,3.5){\resizebox{3.5in}{4.5in}{\includegraphics{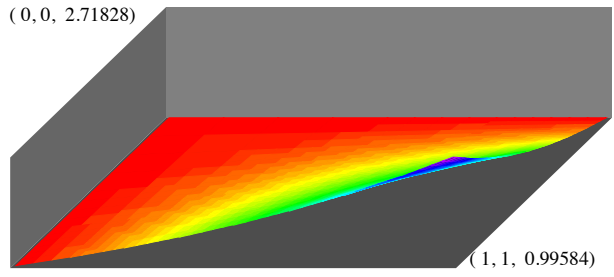}}}
   \put(0, 1.5){\resizebox{3.5in}{4.5in}{\includegraphics{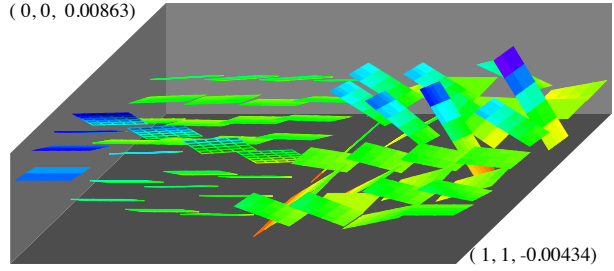}}}
   \put(0,-0.5){\resizebox{3.5in}{4.5in}{\includegraphics{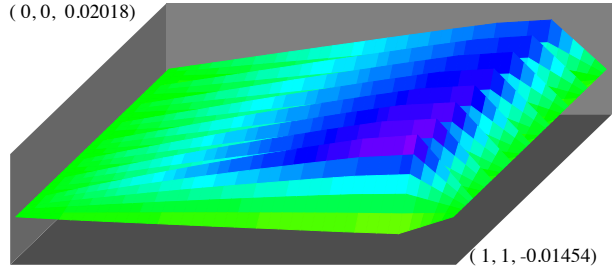}} }
   \put(0,-2.5){\resizebox{3.5in}{4.5in}{\includegraphics{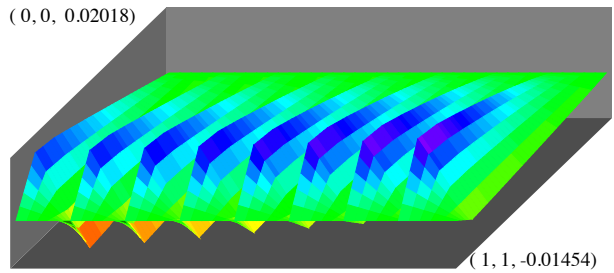}} }
    \end{picture}
\caption{The solution $(\bu)_3$ (the third component) in \eqref{s3},
      and the errors $(\bu-\bu_0)_3$,
    $(\bu-\bu_b)_{3,\bt_1}$ and
    $(\bu-\bu_b)_{3,\bt_2}$ (the two tangential component of the third component)
     on level 4, at $z=0.3$. } \label{f-u}
\end{figure}

Table \ref{table3} contains additional information for the scalar
approximation $p_h$; the fourth column is the error for the scalar
approximation measured at the center of each face in a discrete
$L^2$ fashion. More precisely, for each $q_h=\{q_0, q_b\}\in W_h$,
the semi-norm $\3bar q_h \3bar_{0,h}$ is defined as follows:
$$
\3bar q_h\3bar_{0,h}^2=\sum_{T\in\T_h} h \| q_0 - \Pi
q_b\|_{\partial T}^2,
$$
where $\Pi$ is the average operator on each face. It can be seen
that the convergence in this discrete $L^2$ norm is of order
$O(h^2)$, which is higher than the theoretical prediction.
For this purpose,  we graph the solutions and errors in Figures
  \ref{f-p} and \ref{f-u}.
We believe that some superconvergence is playing a role in the weak
Galerkin finite element method. This is left to interested readers
for an investigation.

The forth test \eqref{s4} was conducted on another solution with
general structure. The goal of this test is to re-confirm the
convergence results developed in earlier Sections. The numerical
results are presented in Table \ref{table4}. The numerical
performance of the weak Galerkin finite element method is similar to
the test case three.

\begin{table}[htb]
  \caption{ \label{table4} The errors,  $\be_h=\bQ_h \bu- \bu_h$ in $H^1$-like norm
  $\3bar \be_h\3bar_1$,
  $\be_0=\bQ_0 \bu - \bu_0$ in $L^2$ norm $\|\be_0\|$,
      $\epsilon_h=Q_hp -p_h$ in $L^2$-like norm $\3bar\epsilon_h\3bar_0$,
$\3bar\epsilon_h\3bar_{0,h}$,
    and $\epsilon_0=Q_0p-p_0$ in $L^2$ norm $\|\epsilon_0\|$,
    for \eqref{s4} by $k=1$ finite elements \eqref{Vh}--\eqref{Wh}.
   And the order $r$ as in $O(h^r)$ of convergence.}
\begin{center}  \begin{tabular}{c|lc|lc|lc|lc|lc}  
\hline grid &  $ \3bar \be_h\3bar_1$&$h^r$ & $
\|\be_0\|_{L^2}$&$h^r$ &  $ \3bar\epsilon_h \3bar_0$&$h^r$ & $
\3bar\epsilon_h \3bar_{0,h}$&$h^r$ &$ \|\epsilon_0\|_{L^2}$&$h^r$   \\
\hline
1   &8.54e0    & -    &1.35e0    &-     &3.60e-1  &-    &3.60e-1  &-    &1.47e-1   &-\\
2   &2.27e0    &1.8   &4.77e-1   &1.5   &2.10e0   &-    &2.08e0   &-    &8.65e-1   &-\\
3   &9.86e-1   &1.2   &1.47e-1   &1.7   &5.17e-1  &2.0  &2.50e-1  &3.1  &1.78e-1   &2.3\\
4   &4.32e-1   &1.1   &3.86e-2   &1.9   &3.09e-1  &0.8  &4.53e-2  &2.5  &4.10e-2   &2.1\\
5   &1.97e-1   &1.1   &9.21e-3   &2.0   &1.71e-1  &0.9  &1.08e-2  &2.0  &1.06e-2   &1.9\\
6   &9.85e-2   &1.0   &2.26e-3   &2.0   &8.75e-2  &1.0  &2.69e-3  &2.0  &2.71e-3   &2.0\\
\hline
\end{tabular}\end{center} \end{table}

In the rest of our numerical experiments, we considered a version of
the finite element scheme (\ref{wg1})-(\ref{wg2}) for which no
convergence theory was developed in the present paper. More
precisely, the WG method makes use of piecewise linear functions for
the vector component $V_h$, but the scalar component is modified as
follows:
$$
W_h=\{w=\{w_0,w_b\}:\{w_0,w_b\}|_T\in P_0(T)\times
P_0(e),e\subset\partial T,w_b=0 \mbox{ on }\partial\Omega\}.
$$
In other words, the scalar variable is approximated by using
piecewise constants on both the interior and the boundary of each
element. Again, it is not known if the current theoretical result
can be extended to this simple WG element, though the numerical
results show an excellent approximation to the exact solution. Table
\ref{table5} contains the numerical results for the test case
\eqref{s3}, and Table \ref{table6} is for the test case \eqref{s4}.

 \begin{table}[htb!]
  \caption{The errors,  $\be_h=\bQ_h \bu- \bu_h$ in $H^1$-like norm,
        $\be_0=\bQ_0 \bu - \bu_0$ in $L^2$ norm,
     $\epsilon_h=Q_hp -p_h$ in $L^2$-like norm,
    and $\epsilon_0=Q_0p-p_0$ in $L^2$ norm,
    for \eqref{s4} by lower order WG finite elements.
   And the order $r$ as in $O(h^r)$ of convergence.}\label{table5}
\begin{center}  \begin{tabular}{c|lc|lc|lc|lc}  
\hline grid &  $ \3bar \be_h\3bar_1$&$h^r$ & $
\|\be_0\|_{L^2}$&$h^r$ &  $ \3bar\epsilon_h \3bar_0$&$h^r$ &
   $ \|\epsilon_0\|_{L^2}$&$h^r$   \\ \hline

1   &6.72e-1  &-     &3.24e-1   &-     &0.00e0    &-      &2.68e-3  &-\\
2   &3.66e-1  &0.9   &8.62e-2   &1.9   &5.73e-3   &-      &2.92e-3  &-\\
3   &1.94e-1  &0.9   &2.09e-2   &2.0   &1.03e-3   &2.5   &6.16e-4  & 2.2\\
4   &1.02e-2  &0.9   &5.07e-3   &2.0   &1.89e-4   &2.5   &1.12e-4  & 2.4\\
5   &5.05e-2  &1.0   &1.26e-3   &2.0   &9.44e-5   &2.0    &2.71e-5  & 2.1 \\
6   &2.52e-2  &1.0   &3.14e-4   &2.0   &4.72e-5   &2.0    &6.78e-6  & 2.0\\
   \hline
\end{tabular}\end{center} \end{table}

\begin{table}[htb!]
  \caption{The errors,  $\be_h=\bQ_h \bu- \bu_h$ in $H^1$-like norm,
        $\be_0=\bQ_0 \bu - \bu_0$ in $L^2$ norm,
     $\epsilon_h=Q_hp -p_h$ in $L^2$-like norm,
    and $\epsilon_0=Q_0p-p_0$ in $L^2$ norm,
    for \eqref{s4} by lower order WG finite elements.
   And the order $r$ as in $O(h^r)$ of convergence.}\label{table6}
\begin{center}  \begin{tabular}{c|lc|lc|lc|lc}  
\hline grid &  $ \3bar \be_h\3bar_1$&$h^r$ & $
\|\be_0\|_{L^2}$&$h^r$ &  $ \3bar\epsilon_h \3bar_0$&$h^r$ &
   $ \|\epsilon_0\|_{L^2}$&$h^r$   \\ \hline

 1& 8.54e0  &-   & 1.35e0  &-   & 3.60e-1 &-   & 1.47e-1&- \\
 2& 2.21e0  &1.9 & 4.27e-1 &1.7 & 2.08e0  &-   & 8.57e-1&- \\
 3& 1.07e0  &1.1 & 1.63e-1 &1.4 & 1.99e-1 &3.4 & 1.23e-1&2.8  \\
 4& 4.35e-1 &1.2 & 3.88e-2 &2.1 & 4.49e-2 &2.1 & 2.71e-2&2.2 \\
 5& 1.96e-2 &1.1 & 9.19e-3 &2.1 & 1.07e-2 &2.1 & 7.23e-3&1.9 \\
 6& 9.82e-2 &1.0 & 2.26e-3 &2.0 & 2.61e-3 &2.0 & 1.85e-3&2.0 \\
   \hline
\end{tabular}\end{center} \end{table}

\end{document}